\documentclass{amsart}     
\usepackage{graphicx}
\usepackage{latexsym}
\usepackage{amssymb}
\usepackage{amsmath}
\usepackage{amsfonts}
\usepackage[mathscr]{eucal}

\usepackage{amscd}

\usepackage{graphicx}
\usepackage{epsfig}

\usepackage{amsthm}
\usepackage[all,cmtip]{xy}
\usepackage{hyperref}
\usepackage{pst-grad} 
\usepackage{pst-plot} 
\usepackage{comment}

\newtheorem{thm}{Theorem}[section]
\newtheorem{cor}[thm]{Corollary}
\newtheorem{lem}[thm]{Lemma}
\newtheorem{prop}[thm]{Proposition}

\newtheorem{soulthm}[thm]{Soul Theorem}
\newtheorem{isotropylem}[thm]{Isotropy Lemma}
\newtheorem{splittingthm}[thm]{Splitting Theorem}

\newtheorem{doublesoulthm}[thm]{Double Soul Theorem}

\newtheorem*{theorem_A}{Theorem A}

\newtheorem*{theorem_B}{Theorem B}

\newtheorem*{theorem_C}{Theorem C}

\theoremstyle{definition}

\newtheorem{example}[thm]{Example}
\newtheorem{rem}[thm]{Remark}

\newtheorem{case}{Case}[section]

\theoremstyle{remark}
\newtheorem{question}[thm]{Question}
\newtheorem*{ack}{Acknowledgments}

\numberwithin{equation}{section}

\newcommand{\RR}{\mathbb{R}}
\newcommand{\CC}{\mathbb{C}}

\newcommand{\KK}{\mathbb{K}}

\newcommand{\klein}{\mathbb{K}}
\newcommand{\MM}{\mathbb{M}}
\newcommand{\RP}{\mathbb{RP}}
\newcommand{\CP}{\mathbb{CP}}
\newcommand{\curv}{\mathrm{curv}}
\newcommand{\torus}{\mathrm{\mathbb{T}}}
\newcommand{\sphere}{\mathrm{\mathbb{S}}}

\newcommand{\cldisc}{\mathrm{\mathbb{D}}}

\newcommand{\Int}{\mathrm{\mathbb{Z}}}
\newcommand{\SO}{\mathsf{SO}}
\newcommand{\SU}{\mathsf{SU}}
\newcommand{\Sp}{\mathsf{Sp}}

\newcommand{\Hh}{\mathsf{H}}
\newcommand{\G}{\mathsf{G}}

\newcommand{\Id}{\mathsf{1}}
\newcommand{\Ss}{\mathsf{S}}

\newcommand{\Fix}{\mathrm{Fix}}

\begin{document}


\title[ Fixed point homogeneous manifolds in low dimensions]{Nonnegatively curved fixed point homogeneous manifolds in low dimensions
	}



\author{Fernando Galaz-Garcia}


\address{ Department of Mathematics, University of Maryland at College Park,
College Park, Maryland, U.S.A.}
\curraddr{Mathematisches Institut, WWU M\"unster, Germany\\
}
\email{f.galaz-garcia@uni-muenster.de}



\begin{abstract}
Let $\G$ be a compact Lie group acting isometrically on a compact Riemannian manifold $M$ with nonempty fixed point set 
$M^\G$. We say that $M$ is \emph{fixed-point homogeneous} if $\G$ acts transitively on a normal sphere to some component of $M^\G$. 
Fixed-point homogeneous manifolds with positive sectional curvature have been completely classified. We classify nonnegatively curved fixed-point homogeneous Riemannian manifolds in dimensions $3$ and $4$ and determine which nonnegatively curved simply-connected $4$-manifolds admit a smooth fixed-point homogeneous circle action with a given orbit space structure. 
\end{abstract}


\maketitle



\section{Introduction}

The study of Riemannian manifolds with nonnegative (sectional) curvature has remained an area of active research in which metric aspects of differential geometry, such as comparison arguments, play a central role (cf. \cite{W,Zi}). Despite the existence of general structure results (e.g., Cheeger-Gromoll \cite{ChGr}) and of obstructions to nonnegative curvature (e.g., Gromov's Betti number theorem \cite{Gr_Betti}), examples of nonnegatively curved manifolds and techniques for their construction are scarce. Thus, finding new examples in this class remains a central problem in the field. In this context, considering manifolds with a ``large''  isometric group action provides  a systematic approach to the study of both positively and nonnegatively curved manifolds (see, e.g., \cite{G}), revealing the structure of these spaces and providing insight into methods for constructing new examples (cf. \cite{GZ}). What we mean by ``large'' is open for interpretation. In this work we will interpret ``large'' as having low \emph{fixed-point cohomogeneity}, which we presently define.

Let $M$ be a compact Riemannian manifold and $\G$ its isometry group, which is a compact Lie group. Observe that $\G$ acts on $M$ by isometries; we will assume that this action is effective. Suppose that $\G$ acts on $M$ with nonempty fixed-point set $M^{\G}$. We define the \emph{fixed-point cohomogeneity} of $M$ as $\dim M/\G - \dim M^{\G}-1 \geq 0$.  
We say  that the action is \emph{fixed-point homogeneous} if the fixed-point cohomogeneity is $0$, i.e., if $M^{\G}$ has codimension $1$ 
in the orbit space $M/\G$. Fixed point homogeneous connected positively curved manifolds were classified by Grove and Searle \cite{GS}. This classification has been proven a strong tool in other classification work on positively curved manifolds with symmetries, e.g, \cite{Wi_sym} and the classification of simply-connected positively curved \emph{cohomogeneity} 1 manifolds \cite{GWZ,Ve} (i.e., positively curved manifolds with an isometric Lie group action whose orbit space is 1-dimensional).

In this work we investigate fixed-point homogeneous Riemannian manifolds with nonnegative curvature. In addition to the intrinsic interest these manifolds have as an extension of the class of positively curved fixed-point homogeneous manifolds studied in \cite{GS}, the classification of these manifolds would likely provide a useful tool in further research, as has been the case for positive curvature.

The presence of an isometric Lie group action provides a link between Riemannian geometry, transformation groups and Alexandrov geometry. In particular,  a fixed-point homogeneous action on a nonnegatively curved manifold $M$ yields information on the structure of $M$. More precisely, if $F$ is a fixed-point set component with maximal dimension, $M$ can be written as the union of $D(F)$, a tubular neighborhood of $F$, and $D(B)$, a neighborhood of a subspace $B\subset M$ determined by the geometry of the action (cf. Section~\ref{S2}). Thus understanding the pieces $D(F)$ and $D(B)$ is a first step in understanding the structure of nonnegatively curved manifolds with a fixed-point homogeneous action. We have focused our attention on dimensions 3 and 4, in which one is able to obtain detailed information on the manifolds and the actions by combining the geometry of the action and the classification results of Orlik and Raymond \cite{ORa,Ra} in dimension 3, and of Fintushel \cite{F1}, in dimension 4. The classification of fixed-point homogeneous $2$-manifolds follows from the classification of fixed-point homogeneous manifolds of cohomogeneity one (cf. Section~\ref{S2}). The only fixed-point homogeneous $2$-manifolds, regardless of curvature assumptions,  are $\sphere^2$ and $\RP^2$. In dimensions 3 and 4 our main results are the following. 

\begin{theorem_A}
Let $M^3$ be a $3$-dimensional nonnegatively curved fixed-point homogeneous Riemannian $\G$-manifold. Then $\G$ can be assumed to be $\SO(3)$ or $\Ss^1$ and $\mathrm{codim}\, M^{\G}=3$ or $2$, respectively. 
\begin{itemize}
	\item[(1)]If 	$\G=\SO(3)$, then $M^3$ is equivariantly diffeomorphic to $\sphere^3$ or $\RP^3$.
	\\
	\item[(2)] If $\G=\Ss^1$, then $M^3$ is equivariantly diffeomorphic to $\sphere^3$, a lens space $L^3$, $\sphere^2\times\sphere^1$, $\RP^2\times\sphere^1$, $\RP^3\#\RP^3$ or the non-trivial bundle $\sphere^2\tilde{\times}\sphere^1$.
\end{itemize}	
\end{theorem_A}

\begin{theorem_B}
Let $M^4$ be a $4$-dimensional nonnegatively curved fixed-point homogeneous $\mathsf{G}$-manifold.  Then $\G$ can be assumed to be $\SO(4)$, $\SU(2)$, $\SO(3)$ or $\Ss^1$.
\begin{itemize}
	\item[(1)] If $\mathsf{G}=\mathsf{SO}(4)$, then $M^4$ is equivariantly diffeomorphic to $\sphere^4$ or $\RP^4$.
	\\
	\item[(2)] If $\mathsf{G}=\mathsf{SU}(2)$, then $M^4$ is equivariantly diffeomorphic to $\sphere^4$, $\RP^4$ or $\mathbb{CP}^2$.
	\\
	\item[(3)] If $\mathsf{G}=\mathsf{SO}(3)$, then $M^4$ is diffeomorphic to a quotient of $\sphere^4$ or $\sphere^3\times\sphere^1$.
	\\
	\item[(4)] If $\mathsf{G}=\mathsf{S}^1$, then $M^4$ is diffeomorphic to a quotient of $\sphere^4$, $\CP^2$, $\sphere^2\times\sphere^2$, $\CP^2\#\pm\CP^2$, $\sphere^3\times\RR$ or $\sphere^2\times\RR^2$. 
\end{itemize}
\end{theorem_B}

Theorems A and B are  proved in Sections 3  and 4. To do this, we completely determine the possible orbit spaces of a fixed-point homogeneous action on a nonnegatively curved $3$- or $4$-manifold. We have provided examples of isometric actions realizing some of the possible orbit space configurations that occur in the proofs. Section~\ref{S2} contains preliminary definitions and results that will be used in subsequent sections. We remark that all of the manifolds in Theorems A and B are known to carry metrics of nonnegative curvature. However, not every $3$-manifold with nonnegative curvature appears in our list, e.g. the Poincar\'e homology sphere, which can be viewed as the quotient space $\SO(3)/\mathsf{I}$, where $\mathsf{I}$ is the icosahedral group. We also point out that, as a consequence of our work, every fixed-point homogeneous nonnegatively curved manifold of dimension $3$ or $4$ decomposes as the union of two disk bundles.

In Section 5 we further study fixed-point homogeneous circle actions on nonnegatively curved simply-connected $4$-manifolds. To put our results in context, let us recall first that, as a consequence of the work of Kleiner \cite{K} and Searle and Yang \cite{SY}, in combination with Fintushel's classification of circle actions on simply-connected 4-manifolds  \cite{F1}  and Perelman's proof of the Poincar\'e 
conjecture, a simply-connected nonnegatively curved $4$-manifold with an isometric circle action is diffeomorphic to $\sphere^4$, $\CP^2$, $\sphere^2\times\sphere^2$ or $\CP^2\#\pm\CP^2$. Let $\chi(M)$ be the Euler characteristic of a manifold $M$. By a well-known theorem of Kobayashi, if $\Ss^1$ acts effectively on $M$, $\chi(M)=\chi(\Fix(M,\Ss^1))$. Thus,  for a simply-connected nonnegatively curved $4$-manifold  $M$ with an isometric 
$\Ss^1$-action, we have $2\leq \chi(M)\leq 4$ and the fixed-point set components are $2$-spheres and isolated fixed-points. Therefore, the only possible fixed-point sets coming from a fixed-point homogeneous circle action on 
$\sphere^4$, $\CP^2$, $\sphere^2\times\sphere^2$ or $\CP^2\#\pm\CP^2$ are
 \[
\Fix(M,\Ss^1)=
	\begin{cases}
		 \sphere^2 				&\text{if $M$ is $\sphere^4$.}\\
		 \sphere^2\cup\{p\} 		&\text{if $M$ is $\CP^2$.}\\
		 \sphere^2\cup\sphere^2 	&\text{if $M$ is $\sphere^2\times\sphere^2$ or $\CP^2\#\pm\CP^2$.}\\
		 \sphere^2\cup\{p',p''\}	&\text{if $M$ is $\sphere^2\times\sphere^2$ or $\CP^2\#\pm\CP^2$.}
	\end{cases} 
\]
Both $\sphere^4$ and $\CP^2$ have metrics of positive curvature with an isometric fixed-point homogeneous circle action, i.e., the fixed-point set of the action is the one in the list above. On the other hand, when $M$ is $\sphere^2\times\sphere^2$ or $\CP^2\#\pm\CP^2$, it is not known if $M$ has a nonnegatively curved Riemannian metric with a fixed point homogeneous circle action realizing each one of the corresponding fixed-point sets listed above. Motivated by this question, in Section~\ref{Section:d4_orbit_spaces} we study smooth fixed-point homogeneous circle actions on $\sphere^4$, $\CP^2$, $
 \sphere^2\times\sphere^2$ or $\CP^2\#\pm\CP^2$.  
 We have summarized our results in the following theorem. We call an $\Ss^1$-action \emph{extendable} if it extends to a $\mathsf{T}^2$-action.

 \begin{theorem_C}
Let $M$ be a simply-connected smooth $4$-manifold with a smooth $\Ss^1$-action. 
	\begin{itemize}
		\item[(1)] If $\mathrm{Fix}(M,\Ss^1)=\sphere^2$, then $M$ is equivariantly diffeomorphic to $\sphere^4$ with a linear action.
		\\
		\item[(2)] If $\mathrm{Fix}(M,\Ss^1)=\sphere^2\cup\{p\}$, then $M$ is equivariantly diffeomorphic to $\pm\CP^2$ with a linear action.
		\\
		\item[(3)] If $\mathrm{Fix}(M,\Ss^1)=\sphere^2\cup\sphere^2$, then $M$ is equivariantly diffeomorphic to $\CP^2\#-\CP^2$ or $\sphere^2\times\sphere^2$ with an extendable action.
		\\
		\item[(4)] If $\Fix(M,\Ss^1)=\sphere^2\cup\{p',p''\}$ and there are no orbits with finite isotropy, then $M$ is equivariantly diffeomorphic to $\CP^2\#\pm\CP^2$ with only one extendable action.
		\\
		\item[(5)] If $\Fix(M,\Ss^1)=\sphere^2\cup\{p',p''\}$ and there is only a weighted arc, then $M$ is equivariantly diffeomorphic to one of the following:
		\\
		\begin{itemize}
			\item[(a)]  $\CP^2\#\CP^2$ with only one extendable action with finite isotropy $\Int_2$.
			\\
			\item[(b)]  $\CP^2\#-\CP^2$ with only one extendable action with finite isotropy $\Int_k$, $k$ odd.
			\\
			\item[(c)]  $\sphere^2\times\sphere^2$ with only one extendable action with finite isotropy $\Int_k$, $k$ even.
		\end{itemize}
	\end{itemize}
 \end{theorem_C}

 Theorem C is an application of Fintushel's classification of circle actions on simply-connected $4$-manifolfds \cite{F1}. It follows from Fintushel's work that  a closed simply-connected smooth $4$-manifold with a smooth $\mathsf{S}^1$-action is diffeomorphic to a connected sum of copies of $\mathbb{S}^4$, $\pm\mathbb{CP}^2$  and $\mathbb{S}^2\times\mathbb{S}^2$. Moreover, the action is determined up to equivariant diffeomorphism by 
a set of orbit space data (cf. Section~\ref{Section:orbit_space}). In our case, the orbit space comes from a fixed-point homogeneous circle action on a nonnegatively curved simply-connected $4$-manifold and has a rather simple structure, which is described in detail in Section~\ref{Section:d4_orbit_spaces}. Parts (1) and (2) of Theorem C are simple corollaries of Fintushel's work. To prove parts (3) and (4) we compute the possible orbit space data and determine the intersection form of $M$ following a recipe given by Fintushel. We get our results by showing that the intersection form obtained from each possible orbit space configuration is equivalent to the intersection form of $\sphere^4$, $\CP^2$, $\sphere^2\times\sphere^2$ or $\CP^2\#\pm\CP^2$.


\begin{ack} The results in this paper are part of the author's dissertation research. The author thanks Karsten Grove, his thesis advisor, for his support and numerous conversations discussing the results contained herein.  The author would also like to thank Ron Fintushel for his help in understanding his work on smooth circle actions on simply-connected 4-manifols \cite{F1,F2}. Finally, the author thanks the Department of Mathematics of the University of Notre Dame, where part of this work was carried out during a two-year stay.  
\end{ack}






\section{Basic setup and tools}
\label{S2}
In this section we introduce some notation and several basic tools
that we will use throughout. We will always assume that our manifolds are connected, unless noted otherwise. 


\subsection{Fixed-point homogeneous manifolds}
\label{S2:FPH_mfds}

Let $\G$ be a compact Lie group acting by isometries on a
compact Riemannian manifold $M$. We will consider the action of
$\G$ as a left action. Given $x\in M$, we denote its \emph{isotropy subgroup} by $\G_x=\{\, g\in \G:gx=x\, \}$
 and the \emph{orbit} of $x$ under the action of $\G$ by $
 \G x=\{\, gx:g\in
\G\, \}\simeq \G/\G_x.
$ We will denote the orbit space of the action by $M/\G$ or $M^*$ and, given a set $A\subset M$, we will denote its image under the projection map $\pi:M\rightarrow M^*$ by $A^*$; for example, the orbit of $x\in M$ will be $x^*$.
Unless mentioned otherwise, we will assume that $\G$ acts
\emph{effectively} on $M$, i.e., that the \emph{ineffective kernel}
$\mathsf{K}=\cap_{x\in M}\G_x$ of the action is trivial. 
Note that the isotropy group $\G_{gx} = g\G_{x}g^{-1}$ is conjugate to $\G_x$ . We say that two orbits $\G x$ and $\G y$ are of the same \emph{type} 
if $\G_x$ and $\G_y$ are conjugate subgroups in $\G$.

We will denote the \emph{fixed-point set} of an element $g\in \G$
by $
M^g=\{x\in M:gx=x\}.
$
The fixed-point set of a subgroup $\Hh\leq
\G$ is $M^\Hh=\cap_{g\in \Hh} M^g$; we will occasionally denote it also by $\mathrm{Fix}(M,\Hh)$. It is well-known that each $M^\Hh$ is
a finite disjoint union of closed totally geodesic submanifolds of
$M$ (cf. \cite{Ko}). Given $M^\Hh$, we define its \emph{dimension} by
$\dim M^\Hh =\max\{\, \dim C_i :C_i \mbox{ is a connected
component of } M^\Hh\, \}.$

Recall that, by the \emph{Slice theorem}, for any $x\in M$, a sufficiently small
tubular neighborhood $D(\G x)$ of $\G x$ is equivariantly
diffeomorphic to $\G\times_{\G_{x}}D_{x}^{\perp}$. Here $D_{x}^{\perp}$ is a ball at the origin of the normal space $T^{\perp}_x$ 
to the orbit $\G x$ at $x$ and $G\times_{\G_{x}}D_{x}^{\perp}$ is the bundle with fiber $D_x^{\perp}$ associated to the principal bundle 
$\G\rightarrow \G/\G_x$.


Suppose now that $\G$ acts on $M$ with non-empty fixed-point set $M^\G$.  We
say that the action is \emph{fixed-point homogeneous} if $M^\G$ has
codimension $1$ in $M^*$; equivalently, if $\G$ acts transitively
on the normal sphere to some component of $M^\G$. We say that $M$ is \emph{fixed-point homogeneous} if it
supports a fixed-point homogeneous action for some compact Lie group $\G$. 


The fact that $\G$ must act transitively on the normal sphere to some component of $M^\G$  determines what Lie groups $\G$ can act fixed-point homogeneously. The groups $\G$ that can
act transitively on a $k$-dimensional sphere $\sphere^k$ with principal isotropy $\Hh$ have been
classified (cf. \cite{Bo1,Bo2,MS,Po}). By possibly replacing $\G$ by a subgroup, it suffices to consider the pairs $(\G,\Hh)$ in the following list. 
Following \cite{GS}, we have labeled each pair $(\G,\Hh)$ by $(a_{k+1}),\ldots,(f)$.
\smallskip
\begin{equation}
\label{L:gps_trans_spheres}
(\G,\Hh)=
\begin{cases}
    (a_{k+1})\ \ \ (\SO(k+1),\SO(k)), & k\geq 1;\\[3pt]
    (b_{m+1})\ \ \ (\SU(m+1),\SU(m)), & k=2m+1\geq 3;\\[3pt]
    (c_{m+1})\ \ \ (\mathsf{Sp}(m+1),\mathsf{Sp}(m)), & k=4m+3\geq 7;\\[3pt]
    (d)\ \ \ \ \ \ \ \,(\G_2,\SU(3)), & k=6;\\[3pt]
    (e)\ \ \ \ \ \ \ \,(\mathsf{Spin}(7),\G_2)), & k=7;\\[3pt]
    (f)\ \ \ \ \ \ \, \,(\mathsf{Spin}(9),\mathsf{Spin}(7)), & k=15.
\end{cases}
\end{equation}
\smallskip

A closed $2$-manifold with a fixed-point homogeneous action must have cohomogeneity one and must be $\sphere^2$ or $\RP^2$ (cf. Corollary~\ref{C:Dim2}). In a curvature free setting, closed $3$-manifolds with a fixed-point homogeneous $\Ss^1$-action have been classified by Raymond \cite{Ra}  (cf. Theorem~\ref{T:3D_FPH_S1_actions}). This is a particular instance of the general Orlik-Raymond-Seifert classification of $3$-manifolds with a smooth $\Ss^1$-action \cite{ORa,Ra,Se} (cf. \cite{Or72}). Fixed-point homogeneous manifolds have also been studied in a Riemannian geometric context. In particular, fixed-point homogeneous Riemannian manifolds with positive sectional curvature have been completely classified up to equivariant diffeomorphism by Grove and Searle (cf. Classification Theorem 2.8 in \cite{GS}).


\subsection{Geometry of the orbit space} 
In this subsection we outline the geometric structure of the orbit space $M^*$ of an isometric Lie group action on a nonnegatively curved compact Riemannian manifold $M$. Such an orbit space is, in general, an Alexandrov space with nonnegative curvature. We start by recalling some basic notions from Alexandrov geometry in the context of an isometric group action (cf. \cite{G}). We will then review some fundamental results linking the geometry of the orbit space $M^*$ with the structure of $M$. 

  Recall that a finite dimensional length space $(X,\mathrm{dist})$ is an \emph{Alexandrov space} if it has curvature bounded from below $\curv \geq k$ (cf. \cite{BBI}).
 When $M$ is a complete, connected Riemannian manifold and $\G$ is a compact Lie group acting (effectively) on $M$ by isometries, the orbit space $M^*$ is equipped with the orbital distance metric induced from $M$, i.e., the distance between $p^*$ and $q^*$ in $M^*$ is the distance between the orbits $\G p$ and $\G q$ as subsets of $M$.  It is well-known that, if $M$ has sectional curvature bounded below $\sec M\geq k$, then the orbit space $M^*$ is an Alexandrov space with $\curv M^* \geq k$. 
 
The \emph{space of directions} $S_{x}X$ of a general Alexandrov space $X$ at a point $x$ is,  by deÞnition, the completion of the 
space of geodesic directions at $x$. The euclidean cone $CS_{x}=T_{x}X$ is called the \emph{tangent space} to $X$ at $x$. In the case of an orbit space $M^*=M/\G$, the space of directions $S_{p^*}M^*$ at a point $p^*\in M^*$ consists of geodesic directions and is isometric to
\[
\sphere^{\perp}_p/\G_p,
\] 
where $\sphere^{\perp}_p$ is the normal sphere to the orbit $\G p$ at $p$.

The possible isotropy groups along a minimal geodesic joining two orbits $\G p$ and $\G q$  in $M$ and, equivalently, along a minimal geodesic joining $p^*$ and $q^*$ in the orbit space $M^*$, are restricted by Kleiner's Isotropy Lemma \cite{K}:


\begin{isotropylem}[Kleiner]
\label{L:isotropy_lemma}

 Let $c:[0,d]\rightarrow M$ be a
minimal geodesic between the orbits $\G c(0)$ and $\G c(d)$. Then, for
any $t\in (0,d)$, $\G_{c(t)}=\G_{c}$ is a subgroup of $\G_{c(0)}$ and
of $\G_{c(1)}$.
\end{isotropylem}

We will also use the following analog of the Cheeger-Gromoll Soul Theorem \cite{ChGr} in the case of orbit spaces. A more general result for Alexandrov spaces with curvature bounded below is due to Perelman \cite{Pe}.


\begin{soulthm}
\label{T:soul_theorem} If $\curv M^* \geq 0$ and $\partial M^*\neq\emptyset$, then there exists a totally convex compact subset $\Sigma\subset M^*$ with $\partial \Sigma=\emptyset$, which is a strong
deformation retract of $M^*$. If $\curv M^*>0$, then $\Sigma=x^*$ is a point, and $\partial M^*$ is homeomorphic to $S_{x^*}M^*\simeq S^{\perp}_{x}/\G_x$.
\end{soulthm}

We end this subsection by recalling the following consequence of the Cheeger-Gromoll Splitting Theorem  (cf. \cite{ChGr2,ChGr}), which we will use repeatedly.


\begin{splittingthm}[Cheeger, Gromoll]
\label{T:splitting_thm} Let $M$ be a compact manifold of nonnegative Ricci curvature. Then $\pi_1(M)$ contains a finite normal subgroup $\Psi$ such that $\pi_1(M)/\Psi$ is a finite group extended by $\Int_1\oplus\cdots\oplus\Int_k$  and $\tilde{M}$, the universal covering of $M$, splits isometrically as $\overline{M}\times\RR^k$, where $\overline{M}$ is compact. 
\end{splittingthm}


\subsection{The orbit space of a fixed-point homogeneous action} 

Recall that the orbit space $M^*$ of a compact nonnegatively curved Riemannian manifold $M$ is a nonnegatively curved Alexandrov space. Moreover, if $M$ is fixed-point homogeneous,  $\partial M^*$ contains a component $F$ of $M^\G$ with maximal dimension. We now carry out the soul construction on $M^*$ and  let $C\subset M^*$ be the set at maximal distance from $F\subset \partial M^*$. Let $B=\pi^{-1}(C)\subset M$ be the preimage of $C$ under the projection map $\pi:M\rightarrow M^*$ . It follows from the Soul Theorem~\ref{T:soul_theorem} that $M$ can be exhibited as the union $M=D(F)\cup_E D(B)$ of neighborhoods $D(F)$ and $D(B)$ along their common boundary $E$. Hence, in the presence of an isometric fixed-point homogeneous $\G$-action, the structure of $M$ is fundamentally linked to $F$ and $B$ and a thorough understanding of the latter yields information on the structure of $M$. This will be our guiding principle. The following theorem, whose proof follows immediately from the proof of Theorem 2 in \cite{SY}, illustrates this philosophy:


\begin{doublesoulthm}\label{T:Double_Soul_Theorem} Let $M$ be a nonnegatively curved fixed-point homogeneous Riemannian $\G$-manifold. If $\mathrm{Fix}(M,\G)$ contains at least two components $X,Y$ with maximal dimension, one of which is
compact, then $M$ is diffeomorphic to an $\sphere^{k+1}$-bundle over $X$, where $\sphere^k=\G/\Hh$,
with $\G$ as structure group. 
\label{thm:codim_2_Sk_bundle}
\end{doublesoulthm}

The following lemma yields information on the distribution of the isotropy groups in the orbit space $M^*$. We refer the reader to \cite{GS_symrank} for a proof.


\begin{lem}
\label{thm:points_ppal_orbits} Let $\G\times M\rightarrow M$ be an isometric fixed-point homogeneous action on a compact nonnegatively curved manifold $M$. Let $C$ be the set at maximal distance from $\partial M^*$. Then all the points in $M^*-\{C\cup
M^\G\}$ correspond to principal orbits.
\end{lem}


We now let $\G\times M^n\rightarrow M^n$ be an isometric fixed-point homogeneous action, with principal isotropy group $\mathsf{H}$, on a compact nonnegatively curved manifold $M^n$ of dimension $n\leq 4$. Let $C$ be the set at maximal distance from $F$, a component of the fixed-point set of the action with maximal dimension. Then $\dim C\leq \dim F \leq 2$. When $C$ has dimension $0$, it follows from the Soul Theorem that $C$ is a point, and the classification of nonnegatively curved fixed-point homogeneous manifolds with this orbit space structure follows immediately from the work of Grove and Searle \cite{GS}.  

When $\dim C=1$, $C^1$ is homeomorphic either to a closed interval $[-1,+1]$ or to a circle. When $C^1$ is a circle, it follows from the Isotropy Lemma~\ref{L:isotropy_lemma} that all the points in the circle have the same isotropy. When $C$ is an interval $[-1,+1]$, this Lemma implies that all the points in the interior of the interval have the same isotropy.  Let $\mathsf{K}_{-}$, $\mathsf{K}_{+}$ and $\mathsf{K}_{0}$ denote, respectively, the isotropy group of points in the subsets $\{\,-1\,\}$, $\{\, +1\,\}$ and $(-1,+1)$ of $C^1\simeq [-1,+1]$. We will refer to this triple as an \emph{isotropy triple} and will denote it by 

\[
\mathsf{K}_{-}\cdots \mathsf{K}_{0}\cdots \mathsf{K}_{+}.
\]
It follows from the Isotropy Lemma~\ref{L:isotropy_lemma} that $\mathsf{K}_0\leq \mathsf{K}_{\pm}\leq \G$.


\begin{rem}\label{rem:Retraction}
A triple $\mathsf{H}\cdots \mathsf{H}\cdots \mathsf{K}$ may occur as the isotropy triple of $C^1\simeq [-1,+1]$. In this case,  the distance function to the endpoint of $C^1$ with isotropy $\mathsf{K}$ has no critical points, so we have a gradient-like vector field whose flow-lines yield a deformation retraction of $M^*$ onto the point with isotropy $\mathsf{K}$, as in the case when $\dim C=0$, in which the field corresponds to the gradient-like vector field of the distance function from $F$ to $C^0$. Hence this case reduces to the case in which $C$ is a point with isotropy $\mathsf{K}$. 
\end{rem}


Nonnegatively curved Alexandrov spaces of dimension 2 appear as orbit spaces of fixed-point homogeneous actions, as well as sets at maximal distance from a boundary component of an orbit space. It is well-known that a 2-dimensional Alexandrov space $X$ is a topological $2$-manifold, possibly with boundary (cf. \cite{BBI}, Corollary 10.10.3). In addition, when $X$ has nonnegative curvature, we have the following result (cf. \cite{ShYa,CaoGe}).


\begin{thm}
\label{T:2d_Alex_spaces} Let $X$ be a $2$-dimensional Alexandrov space of nonnegative curvature. Then, the following hold: $X$ is homeomorphic to either $\RR^2$, $[0,+\infty]\times \RR$, $\sphere^2$, $\RP^2$, $\cldisc^2$, or isometric to $[0,l]\times\RR$, $[0,l]\times \sphere^1(r)$, $[0,+\infty]\times\sphere^1(r)$, $\RR\times\sphere^1(r)$, $\RR\times\sphere^1(r)/\Int_2$, $[0,l]\times\sphere^1/\Int_2$, a flat torus, or a flat Klein bottle for some $l, r>0$.
\end{thm}


\begin{cor}
\label{C:2d_Alex_spaces_wbdry}
A compact $2$-dimensional Alexandrov space with nonnegative curvature and non-empty boundary is homeomorphic to a closed disc $\cldisc^2$ or isometric to a flat M\"obius band $\MM^2$ or a flat cylinder $\sphere^1\times \mathbb{I}$.
\end{cor}


In the rest of this subsection we will assume that $M$ is a closed nonnegatively curved Riemannian manifold with a fixed-point homogeneous isometric $\Ss^1$-action. We will let $F\subset \partial M^*$ be a component of the fixed-point set and $C$ be the set at maximal distance from $F$ in $M^*$. We will study the structure of the orbit space in the case when $\dim F=\dim C$. 


\begin{lem} The only possible isotropy groups in $C$ are $\Id$, $\Int_2$ and $\Ss^1$.
\label{L:orbit_structure_lemma_isotropy}
\end{lem}
\begin{proof}
Let $p^*\in C$ be a point with finite isotropy $\Ss^1_p=\Int_k$, $k\geq 3$. Let $T_p^{\perp}$ be the normal space to the orbit $\Ss^1 p$ at $p$ and let  $F_p=(T_p^{\perp})^{\Ss^1_p}$. We let $F_p^{\perp}$ be the orthogonal complement of $F_p$ in $T_p^{\perp}$. The tangent space $T_{p^*}$ to $M^*$ at $p^*$ can be written as $T_{p^*}\simeq F_p\times (F_p^{\perp})/\Ss^1_p$ and $F_p$ is isomorphic to the tangent space at $p^*$ of the orbit stratum containing $p^*$. Observe that the cone $(F_p^{\perp})/\Ss^1_p$ contains all directions perpendicular to this orbit stratum in $M^*$. Now, let $\gamma$ be a minimal geodesic in $M^*$ joining $p^*$ with $F\subset \partial M^*$. Observe that $\gamma$ is perpendicular to $C$, which has codimension $1$ in $M^*$. Since the orbit stratum containing $p^*$ must be contained in $C$, the direction of $\gamma$  must be contained in $\sphere( F_p^{\perp})/\Ss^1_p$, the quotient of the unit sphere $\sphere (F_p^\perp)$ of $F_p^\perp$ by the isotropy group $\Ss^1_p$. On the other hand, $\sphere(F_p^{\perp})/\Ss^1_p = \sphere(F_p^{\perp})/\Int_k$ has diameter $\pi/2$ so  $\gamma$ cannot be orthogonal to $C$, which has codimension 1 in $M^*$. \end{proof}

We will now consider two cases: $C\subset \partial M^*$ and $\partial M^*=F$.


\begin{lem}
\label{L:orbit_structure_lemma_A} If $C\subset \partial M^*$, then either $C$ is a fixed-point set component or all the points in $C$ have isotropy $\Int_2$. Moreover, $C$ and $F$ are isometric and $M^*$ is isometric to a product $F\times \mathbb{I}$.
\end{lem}
\begin{proof}
A point $p^*$ in $M^*$ is a boundary point if its space of directions $S_{p^*}$ has boundary. Consider the tangent space decomposition $T_{p^*}\simeq F_p\times (F_p^{\perp})/\Ss^1_p$. For $p^*$ to be a boundary point, $\Ss^1_p$ must act transitively on the unit sphere $\sphere (F_p^\perp)$ of $F_p^\perp$ so $\Ss^1_p$ is either $\Ss^1$ or $\Int_2$. Recall that $F_p$ is the tangent space of the orbit stratum of $p^*$ so it follows from the tangent space decomposition that the orbit stratum with $\Ss^1_p$ isotropy is a subset of $C$ of the same dimension. Hence all the points in $C$ must also have isotropy $\Ss^1_p$. The second assertion in the theorem follows from the proof of Theorem 2 in \cite{SY} (cf. Theorem~\ref{T:Double_Soul_Theorem} above). 
\end{proof}


\begin{lem}
\label{L:orbit_structure_lemma_B} Suppose $\partial M^*=F$. 
	\begin{itemize}
		\item[(1)] If $\partial C=\emptyset$, then all the points in $C$ have principal isotropy, $F$ is a double-cover of $C$ and the covering map is a local isometry.\\
		\item[(2)] If $\partial C\neq \emptyset$, then all the points in $\mathrm{int}\, C$ are principal. 
	\end{itemize}
\end{lem}
\begin{proof}
We first prove (1). Let $p^*\in C$ and suppose that $p$ has isotropy group $\Ss^1_p$. Observe that $p^*$ is an interior point of $M^*$. The only possible isotropy groups in $C$ are $\Ss^1$, $\Int_2$ and $\Id$.  Suppose first that $\Ss^1_p=\Int_2$ and consider the tangent space decomposition $T_{p^*}\simeq F_p\times (F_p^{\perp})/\Int_2$. Observe first that $\Int_2$ acts freely on $F^\perp_p$. If $\dim F^\perp_p\geq 2$, then $\mathrm{diam}\, \sphere(F_p^{\perp})/\Int_2=\pi/2$. Let $\gamma$ be a minimal geodesic joining $p^*$ with $F\subset \partial M^*$. Observe that $\gamma$ is perpendicular to $C$, so its direction must be contained in $\sphere(F_p^{\perp})/\Int_2$. Moreover, since $C$ has codimension $1$ in $M^*$, the direction of $\gamma$ is at a distance $\pi/2$ from a codimension 1 subset of $\sphere(F_p^{\perp})/\Int_2$ and it follows that $\sphere(F_p^{\perp})/\Int_2$ is a spherical cone, which implies that $p^*$ is a boundary point, which is a contradiction. If $\dim F^\perp_p=1$, then $\Int_2$ acts transitively on $F^\perp_p$ so $p^*$ is a boundary point, which is a contradiction. Finally, if $\dim F^\perp_p=0$, then the $\Int_2$ orbit stratum has dimension $\dim M -1$. This implies that $\Fix(M,\Int_2)=M$ which contradicts our assumption that the action is effective. If $p^*$ has isotropy $\Ss^1$, then we have $\mathrm{diam}\, \sphere(F_p^{\perp})/\Ss^1=\pi/2$ so $p^*$ must be a boundary point, which is a contradiction.   Hence the only possible isotropy group in $C$ must be $\Id$.
The other assertions follow from the observation that $M^*$ is a manifold with boundary $\partial M^*=F$. Then the Soul Theorem implies that $M^*$ is a line bundle over $C$ and, since $\partial M^*=F$ is connected, it must double-cover  $C$.

To prove part (2), let $p^*$ be a regular point in $C$. Let $\gamma$ be a minimal geodesic from $p^*$ to $F$ and $v$ a tangent vector to $C$ at $p^*$. Parallel translation of $v$ along $\gamma$ is an isometry, since $\curv\geq 0$. In this way we construct a local isometry $\varphi:(C-E^*)\rightarrow F$, where $E^*$ is the set of exceptional orbits. Moreover, this map is an isometry except on $E^*$. Hence $\mathrm{cl}(C-E^*)$ is isometric to a subset of $F$ and, in particular, since $F$ is a manifold, there cannot be any singular points in $\mathrm{int}\, C$. 
\end{proof}


\subsection{Manifolds of cohomogeneity one}
\label{subsec:cohomone_mfds} The classification of fixed-point homogeneous  manifolds of cohomogeneity one follows from the work of Grove and Searle in \cite{GS}.
 The analysis carried out in \cite{GS} also applies when $M$ admits a fixed-point homogeneous cohomogeneity one action, independently of any curvature assumptions. In particular, the following result is an immediate consequence of the method of proof of the  Classification Theorem~2.8 in \cite{GS}.
 
 
 \begin{cor}
 \label{C:cohom1_fph_actions} Let $M$ be a closed, connected Riemannian manifold with a fixed-point homgeneous $\G$-action of cohomogeneity one.
 \\ 
 \begin{itemize}
 	\item[($a_n$)] If $\G=\SO(n)$, then $M$ is $\G$-equivariantly diffeomorphic to $\sphere^n$ or $\RP^n$.
	\\
	\item[($b_n$)] If $\G=\SU(n)$, then $M$ is $\G$-equivariantly diffeomorphic to $\sphere^{2n}$, $\RP^{2n}$ or $\CP^n$.
	\\ 
	\item[($c_n$)] If $\G=\Sp(n)$, then $M$ is $\G$-equivariantly diffeomorphic to $\sphere^{4n}$, $\sphere^{4n}/\Gamma$ $(\Gamma\subset \Sp(1))$, $\CP^{2n}$ or $\mathbb{HP}^n$.
	\\
	\item[($d$)] If $\G=\G_2$, then $M$ is $\G$-equivariantly diffeomorphic to $\sphere^7$ or $\RP^7$.
	\\
	\item[($e$)] If $\G=\mathsf{Spin}(7)$, then $M$ is $\G$-equivariantly diffeomorphic to $\sphere^8$ or $\RP^8$.
	\\
	\item[($f$)] If $\G=\mathsf{Spin}(9)$, then $M$ is $\G$-equivariantly diffeomorphic to $\sphere^{16}$, $\RP^{16}$ or $\mathbf{Ca}\mathbb{P}^2$.
 \end{itemize} 
 \end{cor}

We point out that Hoelscher \cite{Ho} also classified simply-connected, fixed-point homogeneous manifolds, without any curvature assumptions (cf. Proposition~1.23 in \cite{Ho}).

Observe that a $2$-dimensional fixed-point homogeneous manifold must have cohomogeneity one. The classification of these manifolds is then a particular case of Corollary~\ref{C:cohom1_fph_actions}: 

\begin{cor}
\label{C:Dim2}
 Let $M^2$ be a $2$-dimensional fixed-point homogeneous $\G$-manifold.  Then $\G=\Ss^1$ and $M^2$ is equivariantly diffeomorphic to $\sphere^2$ or $\RP^2$.
\end{cor}






\section[Nonnegatively curved FPH $3$-manifolds]{Nonnegatively curved fixed-point homogeneous $3$-manifolds}
\label{S:Dim3}

This section contains the proof of Theorem A. Observe that the orbit space of a fixed-point homogeneous action  on a nonnegatively curved $3$-manifold  is either one- or two-dimensional. In the latter case, we have a circle action and we will make use of the Orlik-Rayomond-Seifert classification of smooth circle actions on $3$-manifolds (cf. \cite{ORa,Ra,Se}). We will briefly recall this classification in the next subsection. We will then prove Theorem~A in Subsection 3.2.
 In Subsection~3.3 we provide examples of isometric actions realizing the orbit spaces appearing in the classification. 



\subsection{Circle actions on 3-manifolds} 
\label{S:circle_actions_on_3manifolds} 
A smooth $\Ss^1$-action on a closed 3-manifold $M$ is completely determined by a \emph{weighted} orbit space (cf. \cite{Or72,ORa}) 
\[
M^*=\{b;(\varepsilon,g,\bar{h},t),(\alpha_1,\beta_1),\ldots,(\alpha_n,\beta_n)\}
\]
which we now describe. The orbit space $M^*$ is a surface of genus $g$ with $0\leq \bar{h}+t$ boundary components. Of these boundary components, $\bar{h}$ correspond to fixed-point set components while $t$ correspond to special exceptional orbits. The symbol $\varepsilon$ takes on the value $o$, when $M^*$ is orientable, and $\bar{n}$ when $M^*$ is non-orientable. There are $n$ exceptional orbits and each one is assigned a pair of integers $(\alpha_i,\beta_i)$ called \emph{Seifert invariants}. These are pairs of relatively prime integers with the property that if $\varepsilon=o$, then $0<\beta_i<\alpha_i$ and if $\varepsilon=\bar{n}$, then $0<\beta_i<\alpha_i/2$. 
We will decribe the Seifert invariants in more detail in the next paragraph.
 If $\varepsilon=o$ and $\bar{h}+t=0$, we let $b$ be an arbitrary integer. If $\bar{h}+t\neq 0$, let $b=0$. If $\varepsilon=\overline{n}$, $\bar{h}+t=0$ and no $\alpha_i=2$, let $b$ take on the values $0$ or $1$, while $b=0$ otherwise. 

We will now describe the Seifert invariants $(\alpha_i,\beta_i)$ (cf. \cite{F1,Or72}). Following the notation in the transformation groups literature, given a set $A\subset M$, we will let $A^*$ denote the projection of $A$ under the orbit map $\pi:M\rightarrow M^*$, so $A^*=\pi(A)$. Let $E$ be the union of the exceptional orbits and suppose  $E^*=\{\, x_1^*,\ldots, x_n^*\,\}$.  For each $x_i^*\in E^*$, let $V_i^*$ be a closed $2$-disk neighborhood such that $V_i^*\cap V_j^* =\emptyset$ if $i\neq j$. For $x_i\in \pi^{-1}(x_i^*)$ there is a closed $2$-disk slice $S_i$ at $x_i$ such that $S_i^*=V^*_i$. We orient $S_i$ so that its intersection number with the oriented orbit $\pi^{-1}(x^*_i)$ is $+1$ in the solid torus $V_i$. This induces an orientation on $m_i$, the boundary of the slice $S_i$. Observe that $m_i$ is null-homotopic in $V_i$. Now let $h_i$ be an oriented principal orbit on $\partial V_i$. Since the action is principal on $\partial V_i$, it admits a cross-section $q_i$. If the isotropy group at $x_i$ is $\Int_{\alpha_i}$, the cross-section $q_i$ of the action on $\partial V_i$ is determined up to homology  by the homology relation $m_i\sim \alpha_i q_i+\beta_i h_i$, where $\alpha_i$ and $\beta_i$ are relatively prime and $0<\beta_i<\alpha_i$. The \emph{Seifert invariants} $(\alpha_i,\beta_i)$ determine $V_i$ up to orientation-preserving equivariant diffeomorphism. If we reverse the orientation of $V_i$, the Seifert invariants become $(\alpha_i,\alpha_i-\beta_i)$.

A fixed-point homogeneous $\Ss^1$-action on a closed $3$-manifold corresponds to having $\bar{h}>0$. The classification of these manifolds is due to Raymond \cite{Ra}.


\begin{thm}[Raymond]
\label{T:3D_FPH_S1_actions}
 Let 
\[
M=\{b;(\varepsilon,g,\bar{h},t),(\alpha_1,\beta_1),\ldots,(\alpha_n,\beta_n)\}
\]
and assume that $\bar{h}>0$, i.e., that $\mathsf{S}^1$ acts on $M$ with fixed points. Then $M$ is diffeomorphic to 
\\
\begin{itemize}
	\item[(1)] $\sphere^3\#(\sphere^2\times\sphere^1)_1\#\cdots\#(\sphere^2\times\sphere^1)_{2g+	\bar{h}-1}\# (\RP^2\times \sphere^1)_1\#\cdots\#(\RP^2\times\#^1)_{t}$ \linebreak$\# L(\alpha_1,\beta_1)\#\cdots\# 	L(\alpha_n,\beta_n)$ if $(\varepsilon, g, \bar{h},t)=(o,g,\bar{h},t)$, $t\geq 0$;
	\\
	\item[(2)] $(\sphere^2\times\sphere^1)_1\#\cdots\#(\sphere^2\times\sphere^1)_{g+\bar{h}-1}\#(\RP^2\times 	\sphere^1)_1\#\cdots\#(\RP^2\times\#^1)_{t}$
 	$\# L(\alpha_1,\beta_1)\linebreak\#\cdots\# L(\alpha_n,\beta_n)$ if $(\varepsilon, g, \bar{h},t)=(\bar{n},g,\bar{h},t)$, $t> 		0$;
	\\
	\item[(3)] $(\sphere^2\tilde{\times}\sphere^1)\#(\sphere^2\times\sphere^1)_1\#\cdots\#(\sphere^2\times	\sphere^1)_{g+\bar{h}-1}$ $\# L(\alpha_1,\beta_1)\#\linebreak \cdots\# L(\alpha_n,\beta_n)$ if $(\varepsilon, g, 	\bar{h},t)=(\bar{n},g,\bar{h},0)$.
\end{itemize}
\end{thm}


\subsection{Proof of Theorem~A}\label{Proof:Theorem_A}

The first assertion follows from the comments at the end of Subsection~\ref{S2:FPH_mfds}. 
When {$(\mathsf{G}, \mathsf{H})=(\SO(3),\SO(2))$}  the conclusion in part (1) follows from Corollary \ref{C:cohom1_fph_actions} in Section 2.


We now prove part (2). When {$(\mathsf{G},\mathsf{H})=(\mathsf{S^1},\mathsf{1})$} the orbit space $M^*$ is a nonnegatively curved $2$-dimensional Alexandrov space with non-empty boundary. Hence $M^*$ is homeomorphic to one of the spaces listed in Corollary \ref{C:2d_Alex_spaces_wbdry}.
Let $F^1\cong \sphere^1$ be a component of the fixed-point set with maximal dimension. Let $C$ be the set at maximal distance from $F^1$  in $M^*$. 
By construction, $\dim C\leq\dim F^1=1$. 
When $\dim C = 0$,   $M^3$ is equivariantly diffeomorphic to $\sphere^3$ or to a lens space (cf. \cite{GS}). When $\dim C=1$, we have $C^1\simeq\sphere^1$ or $C^1\simeq [-1,+1]$.

 Suppose first $C\simeq [-1,+1]$. After another step of the soul construction, we obtain the soul, which must be a point. Then $M^*\simeq \cldisc^2$.
There cannot be points in $C^1$ with $\Ss^1$ isotropy, since the fixed-point set components of a circle action on a $3$-manifold must be circles. Hence the largest isotropy group in the isotropy triple $\mathsf{K}_{-}\cdots \mathsf{K}_{0}\cdots \mathsf{K}_{+}$ of $C^1\simeq [-1,+1]$ is either $\Id$ or  $\Int_q$, for some $q\geq 2$. 
The case where the largest isotropy group is $\Id$ reduces to the case when $C$ is a point with trivial isotropy. By Theorem \ref{T:3D_FPH_S1_actions}, $M^3$ is diffeomorphic to $\sphere^3$. Moreover, it follows from Theorem 1 in \cite{Ra} that $M^3$ must be equivariantly diffeomorphic to $\sphere^3$.


Suppose the largest isotropy group is $\Int_k$, for some $k\geq 2$, so that we have the isotropy triple
$
\Int_{q_{-}}\cdots\Int_{l}\cdots\Int_{q_{+}}.
$
Since the space of directions at a point in $(-1,+1)$ has diameter
$\pi$, it follows that $\Int_l=\Id$. To determine $\Int_{q_{\pm}}$,  let $\gamma$ be a minimal
geodesic from $\partial M^*=\sphere^1$ to $+1\in [-1,+1]\simeq C^1$.
Since $C^1$ is totally convex and $\gamma$ is orthogonal to $C^1$, the space of directions at $+1$ must have diameter at least $\pi/2$.
Hence $\Int_{p_{+}}=\Int_{2}$ or $\Id$. 
Similarly, $\Int_{p_{-}}=\Int_{2}$ or $\Id$.
 Since we have assumed that at least one isotropy group is non-trivial, we have the isotropy triples
\[
	\Id\cdots\Id\cdots\Int_{2}\quad\text{and}\quad  	\Int_{2}\cdots\Id\cdots\Int_{2}.
\]
By Remark~\ref{rem:Retraction}, the first case reduces to the case in which $C$ is a point with isotropy $\Int_2$, so $M$ is diffeomorphic to $\RP^3$. It follows from \cite{Ra} that, up to equivariant diffeomorphism, there is only one action on $\RP^3$ with orbit space a $2$-disk whose boundary is the fixed-point set and a point with $\Int_2$-isotropy in the interior.

In the case of the isotropy triple $\Int_2\cdots\Id\cdots\Int_2$,
 the orbit space $M^*$ is a $2$-disk; its boundary circle is the fixed-point set, and in the interior of the $2$-disk there are two points with $\Int_2$-isotropy. According to Theorem~\ref{T:3D_FPH_S1_actions}, it follows from this orbit space structure that $M^3$ is diffeomorphic to $\mathbb{RP}^3\#\mathbb{RP}^3$. 
We may also read this off the orbit space structure in the following way. 
Divide $M^*$ by a curve $\gamma$ joining different points in the boundary circle so that the two points with $\Int_2$-isotropy lie in different halves of $M^*$. Now observe that $\gamma$ lifts to $\sphere^2$ in $M^3$ and each half of $M^*$ corresponds to $\mathrm{cl}(\RP^3-\mathbb{B}^3)$. Thus $M$ consists of two copies of $\mathrm{cl}(\RP^3-\mathbb{B}^3)$ identified on the boundary sphere. This corresponds to $\RP^3\#\RP^3$. 
Note that $\pi^{-1}(C^1)\cong \RP^2\#\RP^2 \cong \klein^2 \subset M^3$. We can write  $M^3$ as the union of tubular neighborhoods
$D(\sphere^1)$ and $D(\klein^2)$ identified by their common boundary $E^2$, which is an $\sphere^1$-bundle over $\sphere^1$. Since $M$ is orientable we must have that $E^2$ is  $\torus^2$.

According to \cite{Ra}, Theorem 4, there are $4^2=16$ inequivalent actions on $\RP^3\#\RP^3$. We now show that only one of these can occur as an isometric action on a nonnegatively curved $\RP^3\#\RP^3$. Recall that $\RP^3\# \RP^3$ with nonnnegative sectional curvature has $\sphere^2\times\sphere^1$ as a double cover (cf. \cite{Ha}). This in turn has as universal covering space  $\sphere^2\times\mathbb{R}$ with nonnegative curvature. By the Splitting Theorem~\ref{T:splitting_thm}, $\sphere^2\times\sphere^1$  must have a product metric with nonnegative curvature. 
There is only one $\mathsf{S}^1$ action on $\sphere^2\times\sphere^1$ according to \cite{Ra} Theorem~1~(iii). So there is only one $\mathsf{S}^1$-action on $\RP^3\#\RP^3$ by isometries, induced by the action on $\sphere^2\times\sphere^1$. This action is described in Example \ref{dim3:ex_1}.

Suppose $C^1\simeq \sphere^1$. Then $C^1$ is the soul of $M^*$
and all the points in $C^1$ must have the same isotropy. 
Suppose $C^1$ has trivial isotropy. Then $F^1$ double-covers $C^1$ and the orbit space is a M\"obius band $\MM^2$ whose boundary circle is $F^1$.  The orbit space has weights $(\varepsilon, g, \bar{h},t)=(\bar{n},1,1,0)$, and it follows from Theorem~\ref{T:3D_FPH_S1_actions} that $M$ is diffeomorphic to $\sphere^2\tilde{\times}\sphere^1$, the non-trivial $\sphere^2$-bundle over $\sphere^1$. It follows from Theorem 1(iii) in \cite{Ra} that there is only one circle action with fixed points on this manifold. This action is described in Example~\ref{dim3:ex_2}.

Suppose $C^1$ has finite isotropy $\Int_q$. By Lemma~\ref{L:orbit_structure_lemma_B}, we must have $\Int_2$ isotropy and $C^1$ must be a boundary component. In this case the set of special exceptional orbits is $C^1$. 
We have $(\varepsilon, g, \bar{h},t)=(o,0,1,1)$, so $M^3$ is equivariantly diffeomorphic to $\RP^2\times\sphere^1$, according to Theorem~1 in \cite{Ra}. By Theorem~1(iii) in \cite{Ra}, $\RP^2\times\sphere^1$ supports only one circle action with fixed points, up to equivariant diffeomorphism. This action is described in Example~\ref{dim3:ex_3}.

Suppose $C^1$ has isotropy $\Ss^1$. In this case the orbit space is a cylinder whose boundary components correspond to components of the fixed-point set. There are no exceptional orbits. We have $(\varepsilon, g, \bar{h},t)=(o,0,2,0)$ and, by Theorem~1 in \cite{Ra}, $M^3$ is equivariantly diffeomorphic  to $\sphere^2\times\sphere^1$. Moreover, by Theorem~1(iii) in \cite{Ra}, $\sphere^2\times\sphere^1$ supports only one circle action with fixed points, up to equivariant diffeomorphism. This action is described in Example~\ref{dim3:ex_4}.
\qed


\subsection{Examples of isometric fixed-point homogeneous actions on nonnegatively curved $3$-manifolds}

\begin{example}\label{dim3:ex_1}We describe a general construction for $\SO(n-1)$-actions on $\RP^n\#\RP^n$ with orbit space $\cldisc^2$ such that  the boundary circle of $\cldisc^2$ is a fixed-point set component and with two isolated points with $\Int_2$ isotropy in the interior. The action of $\Ss^1\cong\SO(2)$ we want on $\RP^3\#\RP^3$ will then be a particular case of this construction. Observe first that $\RP^n\#\RP^n$ is the quotient of $\sphere^{n-1}\times\sphere^1$ by the $\Int_2$-action given by $-1(x,z)\mapsto (Ax,\bar{z})$, where $A:\sphere^{n-1}\rightarrow\sphere^{n-1}$ is the antipodal map and $z\mapsto \bar{z}$ is complex conjugation when we consider $\sphere^1\subset \mathbb{C}$. Now, consider the $\SO(n-1)$ action on $\sphere^{n-1}\times\sphere^1$ given by letting $\SO(n-1)$ act with cohomogeneity one on $\sphere^{n-1}$ and trivially on $\sphere^1$. Since rotations commute with the antipodal map, this action induces an $\SO(n-1)$-action on the quotient $\RP^{n}\#\RP^{n}$ giving the desired orbit space. Observe also that this induces a $\Int_2$-action on the orbit space of $\sphere^{n-1}\times\sphere^1$, which is a cylinder whose boundary circles are fixed-point components. The quotient of this $\Int_2$-action yields the orbit space of the $\SO(n-1)$-action on  $\RP^n\#\RP^n$, i.e., we have a commutative diagram

\[
	\begin{CD}
	\sphere^{n-1}\times\sphere^1 @>\kappa>> \RP^n\#\RP^n	\\
	@V\pi VV							@V\pi VV 		\\
	\sphere^1 \times \mathbb{I}		@>\kappa >> \cldisc^2
	\end{CD},
\] 
where $\pi$ is the orbit projection map of the $\SO(n-1)$-action and $\kappa$ is the quotient map under the $\Int_2$ covering action.

\end{example}

\begin{example}\label{dim3:ex_2} Let $\Ss^1$ act isometrically on the non-trivial bundle $\sphere^2\tilde{\times}\sphere^1$ with nonnegative curvature by letting $\Ss^1$ act fiberwise with cohomogeneity one. We obtain this action by first considering $\sphere^2\times [0,1]$ with $\Ss^1$ acting by rotations on the first factor and then identifying $\sphere^2\times\{0\}$ with $\sphere^2\times\{1\}$ via the antipodal map, which is an equivariant isometry. The orbit space of the action is  a M\"obius band whose boundary circle is a fixed-point set component
\end{example}

\begin{example}\label{dim3:ex_3}

Let $\Ss^1$ act isometrically on $\RP^2\times\sphere^1$ with nonegative curvature by letting $\Ss^1$ act via the standard cohomogeneity one action on the $\RP^2$ factor and tivially on the $\sphere^1$ factor. The orbit space of this action is a cylinder. One boundary component of the orbit space is a fixed-point set component, while the points in the other boundary component have $\Int_2$ isotropy.
\end{example}

\begin{example}\label{dim3:ex_4}

Let $\Ss^1$ act isometrically on $\sphere^2\times\sphere^1$ with the standard nonnegatively curved product metric by letting $\mathsf{S}^1$ act on the $\sphere^2$ factor via the standard cohomogeneity $1$ action and trivially on the $\sphere^1$-factor.  The orbit space is a cylinder 
whose boundary components correspond to components of the fixed-point set.
\end{example}





\section[Nonnegatively curved FPH $4$-manifolds]{Nonnegatively curved fixed-point homogeneous $4$-manifolds}
\label{section:dimension_4}

In this section we prove Theorem~B. To do so, we will determine the possible orbit spaces of a fixed-pont homogeneous circle action on a nonnegatively curved $4$-manifold $M^4$. In Subsection~\ref{d4:subsection:Examples}  we give examples of isometric actions on $4$-manifolds with nonnegative curvature realizing some of these orbit spaces. In Section~\ref{Section:d4_orbit_spaces}  we will further discuss fixed-point homogeneous circle actions  when $M^4$ is simply-connected.

\subsection{Proof of Theorem B}

The first assertion follows from the comments made at the end of subsection 2.1. Parts (1) and (2) are cohomogeneity one cases and the conclusions follow from Corollary~\ref{C:cohom1_fph_actions} so we need only focus on parts (3) and (4), which we will prove separately.
\\


\noindent \emph{Proof of assertion} (3).
Suppose $(\mathsf{G},\mathsf{H})=(\SO(3),\SO(2))$. 
The orbit space $M^*$ is $2$-dimensional, and hence is homeomorphic to one of the spaces listed in Corollary~\ref{C:2d_Alex_spaces_wbdry}. Observe that  the fixed-point set components are circles. Let $F^1\subset M^\G$ be a component of $\partial M^*$ and let $C$ be the set of points at maximal distance from $F^1$ in the orbit space $M^*$. We have $\dim C\leq \dim F^1=1$. When $\dim C=0$, it follows from \cite{GS} that $M^4$ is equivariantly diffeomorphic to $\sphere^4$ or $\RP^4$, depending on whether the isotropy of $C^0$ is, respectively, $\SO(2)$ or $\mathsf{O}(2)$. 

Suppose $\dim C =1$. We have $C^1 \simeq [-1,+1]$ or $C^1\simeq \sphere^1$.  When $C^1\simeq [-1,+1]$, we have $M^*\simeq \cldisc^2$. Proceeding as in the proof of (2) of Theorem~A (cf. Section~\ref{Proof:Theorem_A}), we see that the only possible isotropy triples $\mathsf{K}_{-}\cdots
\mathsf{K}_0\cdots\mathsf{ K}_{+}$ are
\begin{align*}
 & \SO(2)		\cdots \SO(2)		\cdots	\SO(2),		\\[5pt]
 & \mathsf{O}(2)\cdots  \SO(2)		\cdots \SO(2), 			\\[5pt]
 & \mathsf{O}(2)\cdots    \SO(2)\cdots \mathsf{O}(2). 			
\end{align*} 
The first two cases reduce to the case in which $C$ is a point and $M$ is diffeomorphic, respectively, to $\sphere^4$ and $\RP^4$. In the third case,  $M^4$ can be exhibited as the connected sum of two copies of $\RP^4$. The lift of $C^1\simeq [-1,+1] $ under the projection map $\pi:M\rightarrow M^*$ is $\pi^{-1}([-1,+1])\cong \RP^3\#\RP^3$ so that $M^4$ decomposes as the union of a $3$-disk bundle over $\sphere^1$ and a $1$-disk bundle over $\RP^3\#\RP^3$.  It follows from Example~\ref{dim3:ex_1}, in Section~\ref{S:Dim3}, that this orbit space can be realized by an $\SO(3)$-action on $\RP^4\#\RP^4$, induced from an isometric $\SO(3)$-action on $\sphere^3\times\sphere^1$ .

Suppose $C^1\simeq \sphere^1$. In this case $M^*$ is isometric to a flat cylinder $\sphere^1\times \mathbb{I}$ or to a flat M\"obius band whose boundary is the fixed-point set $F^1$. 

When $M^*$ is a cylinder $\sphere^1\times \mathbb{I}$,  one of the boundary components corresponds to the fixed-point set component $F^1$. The other boundary component, corresponding to $C^1$,  is either another component of the fixed-point set or it has isotropy $\mathsf{O}(2)$. When the boundary is a fixed-point set component, the manifold is an $\sphere^3$-bundle over $\sphere^1$. We can realize this orbit space structure on $\sphere^3\times\sphere^1$ with nonnegative curvature by letting $\SO(3)$ act by cohomogeneity one on the $\sphere^3$-factor and trivially on the $\sphere^1$-factor.   When the boundary has $\mathsf{O}(2)$-isotropy, 
the lift of a geodesic joining two points in different boundary components of $\sphere^1\times\mathbb{I}$ is $\RP^3$. Hence $M^4$ is an $\RP^3$-bundle over the fixed-point set component $\sphere^1$
and $\pi_1(M^4)=\Int_2\times\Int$. By the Splitting Theorem, $M^4$ is covered by $\sphere^3\times\RR$. In fact, we can realize this orbit space structure on $\RP^3\times\sphere^1$ with nonnegative curvature by letting $\SO(3)$ act by cohomogeneity one on the $\RP^3$-factor and trivially on the $\sphere^1$-factor.

Finally, when $M^*$ is a M\"obius band, $M^4$ is an $\sphere^3$-bundle over $C=\sphere^1$. We can realize this orbit space structure on the non-trivial bundle $\sphere^3\tilde{\times}\sphere^1$ with nonnegative curvature by letting $\SO(3)$ act by cohomogeneity one on the $\sphere^3$-fibers.
\\


\noindent \emph{Proof of assertion} (4). Suppose $(\mathsf{G},\mathsf{H})=(\mathsf{S}^1,\mathsf{1})$. Let $F^2\subset \partial M^*$ be a component of the fixed-point set with maximal dimension.
We have $\dim C\leq \dim F^2 =2$. When $\dim C=0$, $M^4$ is equivariantly diffeomorphic to $\CP^2$ when $C^0$ is a fixed point, to $\sphere^4$ when $C^0$ has trivial isotropy, or to $\RP^4$ when $C^0$ has $\Int_2$-isotropy (cf. \cite{GS}). Observe that $C^0$ cannot have isotropy group $\Int_q$, with $q\geq 3$, since the set of points with finite isotropy group $\Int_q$, $q\geq 3$, must have even codimension in $M^4$.


Suppose now that  $\dim C=1$, so that $C^1 \simeq \sphere^1$ or $C^1\simeq[-1,+1]$. We will analyze each case separately. 


Let $C^1\simeq\sphere^1$. Here $C^1$
is the soul of $M^*$ and by the Isotropy Lemma~\ref{L:isotropy_lemma} all the points in $C^1$ must have the same isotropy group. It is well-known that the fixed-point set components of an $\Ss^1$-action have even codimension in $M^4$, so there cannot be isotropy $\Ss^1$ in $C^1$. Hence the largest isotropy group in $C^1\simeq \sphere^1$ is either $\Int_q$, $q\geq 2$, or the trivial subgroup $\Id$.

 Suppose the largest
isotropy group is $\Int_q$, for some $q\geq 2$, so that all the points in $C^1\simeq\sphere^1$ have isotropy $\Int_q$. Observe that there are no critical points for the distance function to $F$ in $M^*-\{F\cup C^1\}$ and we have a gradient-like vector field from $F$ to the soul circle $C^1$ which is radial near $F$ and near $C^1$ (cf. \cite{GS_symrank}). Given a point $p^*$ in $C^1$, the set of flow-lines from $p^*$ to $F$ is a $2$-disk whose lift is a lens space $L(q,q')$.  Hence $M^4$ is a lens space-bundle over $\sphere^1$ and $\pi_1(M^4)\cong \Int_q\times\Int$. By the Splitting Theorem~\ref{T:splitting_thm}, $M^4$ is covered by $\sphere^3\times\RR$.  

The fixed-point set $F^2$ is  diffeomorphic to the boundary of a tubular neighborhood of $C^1$, so it is an $\sphere^1$-bundle over $C^1\simeq \sphere^1$ and hence either a torus $\torus^2$ or a Klein bottle $\klein^2$.  Actions realizing these orbit spaces are described in Examples~\ref{d4:S1_ex_A} and \ref{d4:S1_ex_B}

When the largest isotropy is the principal isotropy group $\Id$, $M^4$ is an $\sphere^3$-bundle over  $\sphere^1$ and $\pi_1(M^4)\cong \Int$. By the Splitting Theorem~\ref{T:splitting_thm}, $M^4$ must be covered by $\sphere^3\times\RR$ equipped with a product metric of nonnegative curvature. The fixed-point set $F^2$ is diffeomorphic to the boundary of a tubular neighborhood $D(C^1)$ in the orbit space $M^*$. Hence $F^2$ must be $\torus^2$ or $\mathbb{K}^2$. The lift  $\pi^{-1}(C^1)$ of $C^1\simeq \sphere^1$ is either $\torus^2$ or $\klein^2$. Recall that $M$ decomposes as the union of $2$-disk bundles 
over $F^2$ and $\pi^{-1}(C^1)$ with common boundary $E^3$. When $M$ is orientable, we must have $F^1\cong \torus^2$, since the fixed-point set component of a smooth $\Ss^1$-action on an orientable manifold is an orientable manifold. It is not difficult to see that, in this case, $\pi^{-1}(C^1)$ must be $\torus^2$. This orbit space structure can be realized by an isometric $\Ss^1$-action on $\sphere^3\times \sphere^1$, as in Example~\ref{d4:S1_ex_C}. 
When $M$ is not orientable, we have $F^2\cong \klein^2$ and it follows that $\pi^{-1}(C^1)$ is $\klein^2$. This orbit space structure can be realized on the non-trivial bundle $\sphere^3\tilde{\times}\sphere^1$, as described in Example~\ref{d4:S1_ex_D}.


Let now $C^1\simeq [-1,+1]$. We first analyze the orbits corresponding to the points in $C^1$.  Suppose the largest isotropy group is $\Id$. This case reduces to the case when $C$ is a point with trivial isotropy and it follows that $M^4$ is diffeomorphic to $\sphere^4$.

Suppose the largest isotropy group is $\Int_q$, for some $q\geq 2$. Proceeding as in the proof of (2) in Theorem~A (cf. Section~\ref{Proof:Theorem_A}), the only possible isotropy triples are 
$\Id\cdots  \Id\cdots \Int_2$  and $\Int_2\cdots  \mathsf{1}\cdots \Int_2$.
The first case reduces to the case when  $C$ is a point with $\Int_2$-isotropy. In this case $M$ is diffeomorphic to $\RP^4$. In the case of the isotropy triple $\Int_2\cdots  \mathsf{1}\cdots \Int_2$, the lift of $C^1$ under the orbit map $\pi:M\rightarrow M^*$ is  $\pi^{-1}([-1,+1])\simeq \RP^2\#\RP^2$.  Since the space of directions at $\pm 1\in C^1$ is $\RP^2$, the boundary of  a tubular neighborhood of $\pm 1$ in $M^*$ is $\RP^2$ and the boundary of a tubular neighborhood of $C^1$ in $M^*$ of $C^1$ is $\RP^2\#\RP^2$. Hence $F^2\cong \RP^2\#\RP^2\cong \klein^2$ and it follows that $M^4$ is non-orientable. Observe that $M^4$ can be written as the union of tubular neighborhoods of $\RP^2\#\RP^2$ and $\RP^2\#\RP^2$ along their common boundary $E^3$. We consider now the orientable double cover $\tilde{M}$ of $M$. The fixed-point set $\tilde{F}^2$ of the lifted isometric circle action  double-covers $F^2\cong\klein^2$ and is orientable, so $\tilde{F}^2\cong \torus^2$. The lift of the set at maximal distance is a circle $\sphere^1$ with no isotropy. This orbit space configuration has been analyzed already 
and it follows that $M$ is covered by $\sphere^3\times\RR$. An isometric $\Ss^1$-action on $\RP^4\#\RP^4$ with this orbit space structure is described in Example~\ref{d4:S1_ex_E}.

Suppose the largest isotropy group is $\mathsf{S}^1$. The possible isotropy configurations are:
\begin{align}
	&  \Ss^1	\cdots	\Id 		\cdots 	\Id, \label{S1:S1-1-1}\\
	&  \Ss^1	\cdots 	\Id 		\cdots 	\Ss^1, \label{S1:S1-1-S1}\\
	&  \Ss^1	\cdots 	\Id 		\cdots 	\Int_{2}, \label{S1:S1-Zl-Zq}\\
	&  \Ss^1	\cdots 	\Int_{l} 	\cdots 	\Ss^1,\text{ for some $l\geq 2$.} \label{S1:S1-Zl-S1}
\end{align}

Case~\ref{S1:S1-1-1} reduces to the case where $C$ is a point with $\Ss^1$-isotropy, so $M^4$ is diffeomorphic to $\CP^2$. In case~\ref{S1:S1-1-S1}, the boundary of a neighborhood of $C^1\simeq [-1,+1]$ in $M^*$ is $\sphere^2$. Hence the $2$-dimensional fixed-point set component $F^2$ is diffeomorphic to $\sphere^2$. Moreover, the lift $\pi^{-1}(C^1)$ is also $\sphere^2$, so we can write $M^4$ as the union of two $2$-disk bundles over $\sphere^2$. Hence $M^4$ is simply-connected and, by a well-known theorem of Kobayashi, $\chi(M^4)=\chi(\mathrm{Fix}(M^4,\mathsf{S}^1))$=4. It follows from Theorem~\ref{T:D4:S1:KSY-classification} in Section~\ref{Section:d4_orbit_spaces} that $M^4$ is diffeomorphic to $\sphere^2\times\sphere^2$ or $\CP^2\#\pm\CP^2$. We will see in Section~\ref{Section:d4_orbit_spaces} that $\CP^2\#\CP^2$ and $\CP^2\#-\CP^2$ are the only simply-connected $4$-manifolds that support smooth circle actions with this orbit space structure. 

In case~\ref{S1:S1-Zl-Zq}, the boundary of a neighborhood of $C^1$ is $\RP^2$ and the lift of $C^1$ in $M^4$ is also $\RP^2$. Hence $F^2\cong \RP^2$, so $M$ is non-orientable and can be written as the union of two $2$-disk bundles over $\RP^2$ glued along their common boundary $E^3$. Let $\tilde{M}$ be the orientable double-cover of $M$ with the lifted isometric circle action. Then the fixed-point set $\Fix(\tilde{M},\Ss^1)$ of the lifted action double-covers the fixed-point set of the $\Ss^1$-action on $M$ and we must have that $\Fix(\tilde{M},\Ss^1)$ consists of a $2$-sphere and two isolated fixed-points. Hence $M^4$ must be double-covered by $\CP^2\#\CP^2$ or $\CP^2\#-\CP^2$ (cf. Case~\ref{S1:S1-1-S1}).

 In case~\ref{S1:S1-Zl-S1}, the boundary of a neighborhood of the interval $C^1$ is $\sphere^2$. Hence the fixed-point set $F^2$ is diffeomorphic to $\sphere^2$. Moreover, the lift of $C^1$ is a manifold, since it is a component of the fixed-point set of $\Int_l$, and corresponds to $\sphere^2$.  As in case~\ref{S1:S1-1-S1}, $M^4$ is diffeomorphic to either $\sphere^2\times\sphere^2$ or $\CP^2\#\pm\CP^2$. Smooth actions with this orbit space structure can be realized on $\sphere^2\times\sphere^2$ and $\CP^2\#\pm\CP^2$ (cf. Section~\ref{Section:d4_orbit_spaces}).  


Suppose $\dim C=2$.  
We consider two cases: $C\subset \partial M^*$ and $\partial M^*=F$.


 Suppose $C^2\subset \partial M^*$. By Lemma~\ref{L:orbit_structure_lemma_A}, $C^2$ is a fixed-point component or all the points in $C^2$ have isotropy $\Int_2$. In both cases $C^2$ is a closed smooth $2$-manifold with nonnegative curvature, $F^2$ and $C^2$ are isometric and $M^*$ is isometric to $F^2\times \mathbb{I}$.  Since $F^2= C^2$ is a closed, nonnegatively curved $2$-manifold, it is diffeomorphic to $\sphere^2$, $\RP^2$, $\torus^2$ or $\klein^2$.

Suppose that $C^2$ is a component of the fixed-point set. By the Double Soul Theorem \ref{T:Double_Soul_Theorem}, $M^4$ is an $\sphere^2$-bundle over $C^2=F^2$,  and $\pi_1(M^4)\cong\pi_1(F^2)$. When $F^2=\sphere^2$, $M^4$ is an $\sphere^2$-bundle over $\sphere^2$ and it follows that $M^4$ is diffeomorphic to $\sphere^2\times\sphere^2$ or $\sphere^2\tilde{\times}\sphere^2\cong\CP^2\#-\CP^2$. Both manifolds support isometric $\Ss^1$-actions with fixed-point set $\sphere^2\cup\sphere^2$ (cf. Example~\ref{d4:S1_ex_F}). 
When $F^2$ is not $\sphere^2$, $M^4$ is not simply-connected.  Let $\tilde{M}^4$ be the universal covering space of $M^4$. Then we have
\[
\tilde{M}^4=
\begin{cases}
	\CP^2\#-\CP^2\text{ or }\sphere^2\times\sphere^2 &\text{ if } F^2=\RP^2;\\
	\sphere^2\times\RR^2 & \text{ if } F^2=\torus^2\text{ or $\klein^2$}.\\
\end{cases}
\]
We can construct examples realizing the orbit space structure $M^*=F^2\times \mathbb{I}$ with the two boundary components corresponding to fixed-point set components by letting $\Ss^1$ act on the product $F^2\times\sphere^2$ by cohomogeneity one on $\sphere^2$ and trivially on $F^2$.

Suppose now that all the points in $C^2$ have isotropy $\Int_2$. Observe that a geodesic from $F^2$ to $C^2$ lifts to $\RP^2$, so $M^4$ is an $\RP^2$-bundle over $F^2$. 
We can construct examples of actions on nonnegatively curved $4$-manifolds with this orbit space structure by considering the product $F^2\times\RP^2$ with $\Ss^1$ acting by cohomogeneity one on $\RP^2$ and trivially on $F^2$.


Suppose now that $\partial M^*=F^2$, so $C^2$ is not a boundary component of $M^*$. We consider two cases, depending on whether or not $C^2$ has boundary. 

Suppose $\partial C^2=\emptyset$. By Lemma~\ref{L:orbit_structure_lemma_B} all the points in $C^2$ have principal isotropy and $F^2$ double-covers $C^2$. Moreover, $M^4$ is an $\sphere^2$-bundle over $C^2$. The only possibilities for $F^2$ are  $\sphere^2$,
$\torus^2$ or $\klein^2$. When $F^2\cong C^2\cong\torus^2$, we construct an example realizing this orbit space structure by considering $(\sphere^2\tilde{\times}\sphere^1)\times\sphere^1$ with $\Ss^1$ acting fixed-point homogeneously on $\sphere^2\tilde{\times}\sphere^1$ and trivially on $\sphere^1$. The orbit space 
is the product of a M\"{o}bius band and $\sphere^1$. This has boundary $\torus^2$, which corresponds to the fixed-point set, and set at maximal distance $\torus^2$.


Suppose $\partial C\neq \emptyset$. Observe that $C$ is a $2$-dimensional Alexandrov space with nonnegative curvature, hence it must be homeomorphic to $\cldisc^2$ or isometric to a flat M\"obius band $\mathbb{M}^2$ or a flat cylinder $\sphere^1\times \mathbb{I}$. By Lemma \ref{L:orbit_structure_lemma_B} there is no isotropy in the interior of $C^2$. A space of directions argument shows that there cannot be an isolated fixed-point in $\partial C^2$, so the only non-trivial isotropy is $\Int_2$.


Assume $C^2=\cldisc^2$. Let us assume first that the largest isotropy group is $\Id$. Then the soul is a point with trivial isotropy and this case reduces to the case in which $C$ is a point with trivial isotropy. In this case $F^2$ is $\sphere^2$ and $M^4$ is diffeomorphic to $\sphere^4$. 

Suppose now that every point in the boundary circle has isotropy $\Int_2$. It follows from the Orlik-Raymond classification of $3$-manifolds with a smooth $\Ss^1$-action that the lift of $C^2$ is $\sphere^2\tilde{\times}\sphere^1$, the non-trivial $\sphere^2$-bundle over $\sphere^1$. Then we can write $M^4$ as the union of disk bundles over $\sphere^2$ and $\sphere^2\tilde{\times}\sphere^1$ glued along their common boundary $E^3$. 
We must have $E^3\simeq \sphere^2\times\sphere^1$, so $\pi_1(M^4)\cong \Int_2$. Hence  $M^4$ is double-covered by $\sphere^4$, $\CP^2$, $\sphere^2\times\sphere^2$ or $\CP^2\#\pm\CP^2$. Since $\Fix(M^4,\Ss^1)=\sphere^2$, $M^4$ must be orientable. Since $\chi(M^4)=\chi(\Fix(M^4,\Ss^1))=2$,  $M^4$ is a quotient of $\sphere^2\times\sphere^2$ (cf. \cite{Hi}, Ch. IX, Lemma 3).


Finally, suppose that there are isolated points in $\partial C^2$ with finite isotropy $\Int_2$. By compactness there are finitely many of these points in the boundary circle. Denote them by $\overline{p}_1, \ldots, \overline{p}_k$, for some $k\geq 1$.
 We will now show that there can be at most two isolated points with $\Int_2$-isotropy in $\partial C^2$. Let $\bar{q}$ be a point in the interior of $C^2$ and let $\gamma_1, \ldots,\gamma_k$ be minimal geodesics joining $\bar{q}$ with $\bar{p}_1,\ldots,\bar{p}_k$. Since $C^2$ is totally geodesic, these geodesics are contained in $C^2$. Now, observe that $C^2$ deformation retracts onto $U=\cup_{i=1}^k\gamma_1$. Hence a tubular neighborhood $D(C^2)$ is homotopy equivalent to a tubular neighborhood $D(U)$. The boundary of $D(U)$ is the connected sum of $k$ projective spaces and $\partial D(C^2)\cong F^2$ is homotopy equivalent to $\partial D(U)$. Hence $F^2$, which is a closed $2$-manifold with nonnegative curvature, is homotopy equivalent to a connected sum of $k$ projective spaces. Hence $\pi_1(F^2)\cong \pi_1(\#_{i=1}^k\RP^2)$ and we must have $k=1$ or $2$. When we have only one isolated point with $\Int_2$-isotropy, this case reduces to the case in which $C$ is a point with $\Int_2$ isotropy and hence $M^4$ is diffeomorphic to $\RP^4$. When there are two points with $\Int_2$-isotropy, this case reduces to the case when $C$ is an interval with endpoints with $\Int_2$-isotropy. In this case the manifold is diffeomorphic to $\RP^4\#\RP^4$.


Suppose $C^2=\sphere^1\times \mathbb{I}$. The possible isotropy groups are $\Int_2$ or the trivial group $\Id$.
Suppose the largest isotropy group is $\Id$. Then $M^*$ is a manifold with totally geodesic boundary and soul $\sphere^1$. This case reduces to the case in which $C=\sphere^1$ with trivial isotropy and the manifold is then diffeomorphic to $\sphere^3\times\sphere^1$ or $\sphere^3\tilde{\times}\sphere^1$.


Suppose now the largest isotropy group is $\Int_2$. Since $\sphere^1\times\mathbb{I}$ has the product metric, the boundary components are closed geodesics. It follows from the Isotropy Lemma that, if a point in a boundary circle of $\sphere^1\times\mathbb{I}$ has isotropy $\Int_2$, then every point in this circle has isotropy $\Int_2$. Assume first that there are two boundary components with $\Int_2$-isotropy. 
Observe that $F^2\cong \torus^2$ and the lift of $C^2\cong \sphere^1\times \mathbb{I}$ is $\RP^2\#\RP^2\times\sphere^1\cong\KK^2\times\sphere^1$. Then $M^4$ is the union of tubular neighborhoods $D(\torus^2)$ and $D(\KK^2\times\sphere^1)$ glued along their common boundary. 
It follows from Van-Kampen's Theorem that $\pi_1(M^4)\cong\pi_1(\RP^3\#\RP^3)\times\Int$. It follows from the Spitting Theorem that  $M^4$ is covered by $\sphere^2\times\RR^2$.
This orbit space structure can be realized on $\RP^3\#\RP^3\times\sphere^1$ with $\Ss^1$ acting fixed-point homogeneously on the first factor and trivially on the second factor. 

Suppose we only have one boundary component with finite isotropy. 
This case reduces to the case in which $C$ is a circle with $\Int_2$-isotropy. Hence $M^4$ is diffeomorphic to an $\RP^3$-bundle over $\sphere^1$.


Suppose $C^2=\mathbb{M}^2$. Suppose the largest isotropy group is $\Id$. Then the soul is $\sphere^1$ and this case reduces to the case in which $C=\sphere^1$ with trivial isotropy. It follows that $M^4$ is diffeomorphic to $\sphere^3\times\sphere^1$ or $\sphere^3\tilde{\times}\sphere^1$.

Suppose now that the largest isotropy group is $\Int_2$. We have isotropy $\Int_2$ on all the points in the boundary of $C^2$ and the lift of $C^2$ is $\klein^2\tilde{\times}\sphere^1$, a non-trivial $\klein^2$-bundle over $\sphere^1$ . Then $M^4$ is the union of tubular neighborhoods $D(\klein^2)$ and $D(\KK^2\tilde{\times}\sphere^1)$ glued along their common boundary $E^3$. 
Now, since $F=\klein^2$, $M^4$ must be non-orientable. Passing to the orientable double-cover $\tilde{M}^4$, we must have $\tilde{F}=\torus^2$, and it follows from the previous case that $M^4$ is covered by $\sphere^2\times\torus^2$.

\qed


\begin{rem}
By the Splitting Theorem, any isometric action on $\sphere^3\times\RR$ must split, acting by isometries on each factor. There is only one isometric fixed-point homogeneous action on $\sphere^3$ up to equivariant diffeomorphism (cf. Section~\ref{S:Dim3}) so  there is only one fixed-point homogeneous isometric action on a quotient of $\sphere^3\times\RR$ with nonnegative curvature. Similarly, there is only one fixed-point homogeneous isometric action on a quotient of $\sphere^2\times\RR^2$ with nonnegative curvature. 
\end{rem}


\subsection{Examples of isometric fixed-point homogeneous actions on nonnegatively curved $4$-manifolds}\label{d4:subsection:Examples}


\begin{example}\label{d4:S1_ex_A} 

The fixed-point homogeneous $\Ss^1$-action on the round $3$-sphere $\sphere^3$ commutes with the $\Int_q$ action whose quotient is the lens space $L(q,q')$. Hence the covering map $\kappa:\sphere^3\rightarrow L(q,q')$ induces a fixed-point homogeneous $\Ss^1$-action on $L(q,q')$ whose orbit space is a $2$-disk whose boundary circle is the fixed-point set of the action and whose set at maximal distance is a point with finite isotropy $\Int_q$. Consider now the $\Ss^1$-action on $L(q,q')\times\sphere^1$, equipped with the product metric, given by letting $\Ss^1$ act fixed-point homogeneously on $L(q,q')$ and trivially on $\sphere^1$. The orbit space is a solid torus $\cldisc^2\times\sphere^1$ with $F^2=\torus^2$ and $C^1\simeq\sphere^1$ with $\Int_q$ isotropy.

\end{example}


\begin{example}\label{d4:S1_ex_B} 
Consider $L(q,q')\tilde{\times} \sphere^1 \cong (L(q,q')\times [0,1)]/(x,0)\sim(Ax,1)$ where $A$ is the map induced on $L(q,q')$ by the antipodal map on $\sphere^3$ via the covering map $\kappa:\sphere^3\rightarrow L(q,q')$.  Since $A:L(q,q')\rightarrow L(q,q')$ commutes  with the fixed-point homogeneous $\Ss^1$-action on $L(q,q')$, we have a fixed-point homogeneous action on $L(q,q')\tilde{\times}\sphere^1$ by letting $\Ss^1$ act-fixed point homogeneously on the $L(q,q')$-fibers. The orbit space is a non-trivial $\cldisc^2$-bundle $\cldisc^2\tilde{\times}\sphere^1$ whose boundary $F^2=\klein^2$ is the fixed-point set and $C^1$ is a circle with $\Int_q$ isotropy.

\end{example}


\begin{example}\label{d4:S1_ex_C}

Let  $\Ss^1$ act isometrically on $\sphere^3\times \sphere^1$, equipped with the standard nonnegatively curved product metric, by  taking the fixed-point homogeneous $\Ss^1$-action on the $\sphere^3$ factor and letting $\Ss^1$ act trivially on the $\sphere^1$ factor. The orbit space is a solid torus $\cldisc^2\times\sphere^1$ whose boundary $\torus^2$ is the fixed-point set of the action, and the set at maximal distance is a circle with trivial isotropy. 

\end{example}


\begin{example}\label{d4:S1_ex_D} 

Let $\Ss^1$ act on the non-trivial bundle $\sphere^3\tilde{\times}\sphere^1$ by taking the fixed-point homogeneous action on each fiber. The orbit space is the non-trivial $\cldisc^2$-bundle over $\sphere^1$, whose boundary $\klein^2$ corresponds to the fixed-point set $F^2$, with set at maximal distance $C^1\simeq \sphere^1$ with trivial isotropy. Let us denote this pair  by $[F, C]^*$ and  its lift by $[F,B]$, so that in this case we have $[F,C]^*=[\klein^2,\sphere^1]$ and $[F,B]=[\klein^2,\klein^2]$.
 Observe that this action on $\sphere^3\tilde{\times}\sphere^1$ is induced by the action of $\Ss^1$ on $\sphere^3\times\sphere^1$ via the double-covering map $\kappa:\sphere^3\times\sphere^1\rightarrow\sphere^3\tilde{\times}\sphere^1$ and we have a commutative diagram
\[
	\begin{CD}
	\sphere^3\times\sphere^1 @>\kappa>> \sphere^3\tilde{\times}\sphere^1	\\
	@V\pi VV							@V\pi VV 		\\
	\cldisc^2 \times \sphere^1		@>\kappa >> \cldisc^2\tilde{\times}\sphere^1
	\end{CD},
\] 
where $\pi$ is the orbit projection map of the $\Ss^1$-action and $\kappa$ is the quotient map under the $\Int_2$ covering action.

\end{example}


\begin{example}\label{d4:S1_ex_E} We describe an isometric $\Ss^1$-action on $\RP^4\#\RP^4$ with fixed-point set $\RP^2\#\RP^2$ and set at maximal distance $C^1\simeq[-1,+1]$ with endpoints having isotropy $\Int_2$. Observe first that  $\RP^4\#\RP^4$ is a quotient of $\sphere^3\times\sphere^1\subset \mathbb{C}^2\times\mathbb{C}$ by the action of $\Int_2$ given by 
\[
-1((z_1,z_2),z_3)\mapsto ((-z_1,-z_2),\bar{z}_3),
\]
i.e., $\Int_2$ acts by the antipodal map on $\sphere^3\subset \CC^2$ and by conjugation on $\sphere^1\subset \CC$. On $\sphere^3\subset \CC^2$ we have the standard fixed-point cohomogeneity $\Ss^1$-one action given by
\[
\lambda(z_1,z_2)\mapsto (\lambda z_1,z_2),\quad \lambda \in \Ss^1,\ (z_1,z_2)\in \sphere^3.
\]
The fixed-point set of this action is a circle. We extend this action to a fixed-point homogeneous action on $\sphere^3\times\sphere^1$ by letting $\Ss^1$ act fixed-point homogeneously on the $\sphere^3$-factor and trivially on the $\sphere^1$-factor. Since the $\Ss^1$-action  on $\sphere^3\times\sphere^1$ commutes with the $\Int_2$-action, we have an induced $\Ss^1$-action on $\RP^4\#\RP^4$. Moreover, the orbit space $(\sphere^3\times\sphere^1)^*\simeq \cldisc^2\times\sphere^1$ double-covers the orbit space $(\RP^4\#\RP^4)^*$. The fixed-point set of the induced $\Ss^1$-action on $\RP^4\#\RP^2$ is $\RP^2\#\RP^2$ and the set at maximal distance is $C\simeq [-1+1]$, with endpoints having $\Int_2$ isotropy.

\end{example}


\begin{example}\label{d4:S1_ex_F} 
We describe $\Ss^1$-actions on $\sphere^2\times\sphere^2$ and $\sphere^2\tilde{\times}\sphere^2\cong\CP^2\#-\CP^2$ with fixed-point set $\sphere^2\cup\sphere^2$. On $\sphere^2\times\sphere^2$ let $\Ss^1$ act by cohomogeneity one on the first $\sphere^2$ factor and trivially on the second $\sphere^2$ factor. To obtain an isometric $\Ss^1$-action on $\CP^2\#-\CP^2$ with nonnegative curvature and fixed-point set $\sphere^2\cup\sphere^2$ start by letting $\Ss^1$ act fixed-point homogeneously on $\CP^2$. This action has fixed-point set $\sphere^2\cup \{p\}$.  We remove an invariant neighborhood of the isolated fixed point and do the same construction on $-\CP^2$ equipped with a fixed-point homogeneous $\sphere^1$-action. Now take an equivariant connected sum to obtain $\CP^2\#-\CP^2$ with nonnegative curvature and a fixed-point homogeneous isometric $\Ss^1$-action with fixed-point set $\sphere^2\cup\sphere^2$.

\end{example}






\section[Fixed-point homogeneous circle actions]{Fixed-point homogeneous circle actions on nonnegatively curved simply-connected $4$-manifolds}
\label{Section:d4_orbit_spaces}



\subsection{Introduction} 
Effective, locally smooth circle actions on $4$-manifolds  were classified up to equivariant homeomorphism  by Fintushel in \cite{F1,F2}. This classification holds in the smooth category, as a result of carrying out the constructions therein in this setting \cite{FSS09}. In particular, as an immediate consequence of Fintushel's results, work of Pao \cite{Pa78}, and the validity of the Poincar\'e conjecture due to Perelman  \cite{P1,P2,KL,MT} one has the following theorem (cf. Theorem 13.2 in \cite{F2}).

\begin{thm}
\label{T:D4:S1:Fintushel_classification}
Let $M$ be a closed simply-connected smooth $4$-manifold with a smooth $\mathsf{S}^1$-action. Then $M$ is diffeomorphic to a connected sum of copies of $\mathbb{S}^4$, $\pm\mathbb{CP}^2$  and $\mathbb{S}^2\times\mathbb{S}^2$. Moreover, the action is determined up to equivariant diffeomorphism by 
so-called legally weighted
orbit space data. 
\end{thm}

Suppose now that $M$ is a simply-connected Riemannian $4$-manifold with an isometric $\Ss^1$-action. If $M$ has positive curvature, it follows from the work of Kleiner and Hsiang \cite{HK} that the Euler characterstic of $M$, denoted by $\chi(M)$, is $2$ or $3$. More generally, if $M$ has nonnegative curvature, it follows from the work of Kleiner \cite{K} or of Searle and Yang \cite{SY} that  $2\leq \chi(M)\leq4$. Combining these  facts with Theorem~\ref{T:D4:S1:Fintushel_classification} yields the following result. 

\begin{thm}
\label{T:D4:S1:KSY-classification} Let $M$ be a compact, simply-connected Riemannian 4-manifold with an isometric $\mathsf{S}^1$-action. 
\begin{itemize}
\item[(1)] If $M$ has positive curvature, then $M$ is diffeomorphic  to $\mathbb{S}^4$ or $\mathbb{CP}^2$.\\

\item[(2)] If $M$ has nonnegative curvature, then $M$ is diffeomorphic  to $\mathbb{S}^4$, $\mathbb{S}^2\times \mathbb{S}^2$, $\mathbb{CP}^2$ or $\mathbb{CP}^2\# \pm \mathbb{CP}^2$.
\end{itemize}
\end{thm}

 In section~\ref{Section:computations} we apply Fintushel's work \cite{F1} to prove our third main result, Theorem~C in the introduction, obtaining further information on the orbit space of a smooth fixed-point homogeneous $\Ss^1$-action on a nonnegatively curved simply-connected Riemannian manifold $M$. We will use the orbit space data to identify $M$ using the recipe given in \cite{F1} for computing its intersection form. We have collected in Section~\ref{Section:orbit_space} the definitions and results 
from \cite{F1} that we use in section~\ref{Section:computations} to obtain our results.
 
The classification of positively curved fixed-point homogeneous manifolds due to Grove and Searle \cite{GS}, which does not require the Poincar\'e conjecture, implies that a compact, simply-connected Riemannian 4-manifold with positive curvature and an isometric fixed-point homogeneous $\mathsf{S}^1$-action must be equivariantly diffeomorphic to $\sphere^4$ or $\CP^2$ with a linear action. More generally, a conjecture of Grove states that this should be the case for any isometric $\Ss^1$-action on a positively curved simply-connected Riemannian manifold (cf. \cite{G09}). It is an interesting question whether or not an analogous conjecture also holds  for nonnegatively curved manifolds. In this more general case, we will say that an $\Ss^1$-action is \emph{extendable} if it extends to a $\mathsf{T}^2$-action. Observe that a smooth linear $\Ss^1$-action on $\sphere^4$ or $\CP^2$ is extendable. On the other hand, it follows from work of Orlik and Raymond \cite{ORaT2} that a smooth extendable $\Ss^1$-action on $\sphere^4$ or $\CP^2$  is equivariantly diffeomorphic to a linear action. Thus, for smooth $\Ss^1$-actions on $\sphere^4$ or $\CP^2$ the notions of linearity and extendability coincide. This motivates the following question. 

 \begin{question}
 \label{question} \emph{Is an isometric $\Ss^1$-action on a simply-connected nonnegatively curved $4$-manifold equivariantly diffeomorphic to a linear action on $\sphere^4$ or  $\CP^2$, or to an extendable action on $\sphere^2\times\sphere^2$ or $\CP^2\#\pm\CP^2$?}
 \end{question}

We will see in section~\ref{Section:orbit_space} that the answer to this question is \emph{yes}, provided the $\Ss^1$-action is fixed-point homogeneous. This will be a simple consequence of \cite{F1} and our work in Section $4$.

\begin{rem}[Added in proof] Grove's conjecture that an isometric $\Ss^1$-action on a positively curved  $\sphere^4$ or $\CP^2$ must be equivariantly diffeomorphic to a linear action has been confirmed by Grove and Wilking \cite{GW}. More generally, they have answered Question~\ref{question} affirmatively. 
\end{rem}
 

\subsection{Fintushel's construction}
\label{Section:orbit_space}

Let $M$ be a simply-connected $4$-manifold with a smooth $\Ss^1$-action with orbit space $M^*$. In this section we review the definitions and results from \cite{F1} that we will use in the next section to prove Theorem~C. 

\subsubsection{The weighted orbit space} Let us recall first some basic facts and terminology from \cite{F1} pertaining to the orbit space $M^*$. We will denote the fixed-point set by $F$, the set of exceptional orbits by $E$ and the set of principal orbits by $P$. Given a subset $X\subset M$, we will denote its projection under the orbit map $\pi: M\rightarrow M^*$ by $X^*$. Given a subset $X^*\subset M^*$, we will let $X=\pi^{-1}(X^*)$ be its preimage under $\pi$. The orbit space $M^*$ is a simply-connected $3$-manifold with $\partial M^*\subset F^*$, the set $F^*-\partial M^*$ of isolated fixed points is finite and $F^*$ is nonempty. The components of $\partial M^*$ are $2$-spheres and the closure of $E^*$ is a collection of polyhedral arcs and simple closed curves in $M^*$. The components of $E^*$ are open arcs on which orbit types are constant, and these arcs have closures with distinct endpoints in $F^*-\partial M^*$. We will reserve the term  \emph{regular neighborhood} of  $X^*\subset E^*\cup F^*$  for those regular neighborhoods $N^*$ of $X^*$ that satisfy $N^*\cap (E^*\cup F^*)=X^*$.

 We remark that, if we do not require that $M$ be simply-connected, we may have loops $Q^*\subset E^*$. Consider, for example, the $\Ss^1$-action on $\RP^3\times\sphere^1$ given by the fixed-point homogeneous action of $\Ss^1$ on $\RP^3$, induced by the fixed-point homoeneous $\Ss^1$-action on $\sphere^3$ via the covering map, and the trivial action on the $\sphere^1$-factor. In this case $M^*$ is a solid torus with $Q^*=E^*$ a loop with $\Int_2$ isotropy.  

The orbit space $M^*$ is assigned a set of data, called \emph{weights}, which we now describe.
\\

(a) Let $F_i^*$ be a boundary component of  $M^*$, choose a regular neighborhood $F^*_i\times[0,1]$ and orient $F_i^*\times 1$ by the normal out of $F_i^*\times[0,1]$. The restriction of the orbit map gives a principal $\Ss^1$-bundle over $F^*_i\times 1$ and $F^*_i$ is assigned the Euler number of this bundle.  This is independent of the choice of the collar. We will call $F^*_i$ a \emph{weighted sphere}.

(b) If $x^*$ is an isolated fixed point, i.e., if $x^*\in F^*-(\partial M^*\cup \mathrm{cl}\, E^*)$, let $B^*$ be a polyhedral $3$-disk neighborhood of $x^*$ with $B^*-x^*\cup P^*$. We obtain a principal $\Ss^1$-bundle over $\partial B^*$ with total space $\sphere^3$ by restricting the orbit map. Orient $\partial B^*$ by the normal out of $B^*$ and assign to $x^*$ the Euler number, $\pm 1$, of the bundle. 

(c) Let $L^*$ be a simple closed curve in $E^*\cup F^*$. To each component $J^*$ of $E^*$ in $L^*$ we assign Seifert invariants (cf. Section 3, Section \ref{S:circle_actions_on_3manifolds}) in the following way. Fix an orientation on $L^*$. This induces an orientation on each component $J^*$ of $E^*$ in $L^*$. Let $y^*$ be an endpoint of $\mathrm{cl} J^*$ and let $B^*$ be a polyhedral $3$-disk neighborhood of $y^*$ such that $B^*\cap (E^*\cup F^*) =B^*\cap L^*$ is an arc and $B^*\cap F^* = y^*$. If $\partial B^*$ is oriented by the normal with direction $J^*$ then $\partial B$ is an oriented $3$-sphere. Assign to $J^*$ the Seifert invariants $(\alpha,\beta)$ of the orbit in $\partial B$ with image in $J^*$. The covering homotopy theorem of Palais implies that this definition is independent of the choices made. 

The weights assigned to $L^*$ consist of the orientation and the Seifert invariants. We abbreviate this system of weights by $\{\, (\alpha_1,\beta_1),\ldots,(\alpha_n,\beta_n)\,\}$, where the order of the $(\alpha_i,\beta_i)$ is determined up to a cyclic permutation, and we call $L^*$ a \emph{weighted circle}. If the orientation of $L^*$ is reversed, each $(\alpha_i, \beta_i)$ becomes $(\alpha_i, \alpha_i-\beta_i)$ and we regard the resulting weighted circle as equivalent to the first. 

(d) Let $A^*$ be an arc which is a component of $E^*\cup F^*$. Orient $A^*$ and assign Seifert invariants as in (c). Let $y^*$ be the initial point or final point of $A^*$ and $B^*$ a small $3$-disk neighborhood of $y^*$. Proceeding as in (c), $\partial B$ has the $\Ss^1$-action $\{b; (o,0,0,0);(\alpha,\beta)\}$ (cf. Section 3, Section \ref{S:circle_actions_on_3manifolds}). Assign this integer $b$ to $y^*$. We call $A^*$ a \emph{weighted arc} and write the weight system as $[b';(\alpha_1,\beta_1),\ldots,(\alpha_n,\beta_n);b'']$. Reversing the orientation on $A^*$ changes the weight system to $[-1-b'';(\alpha_n,\alpha_n-\beta_n),\ldots,(\alpha_1,\alpha_1-\beta_1);-1-b'']$ which we regard as equivalent to the original weight system of $A^*$.
We also recall the following Lemma (cf. Lemma 3.5 in \cite{F1}).

\begin{lem}\label{L:Fintushel_lemma} \emph{(a)} If $(\alpha_i,\beta_i)$ and $(\alpha_{i+1,\beta_{i+1}})$ are the Seifert invariants assigned to adjacent arcs in some weighted arc or circle, then
\[
\begin{vmatrix}
\alpha_i	& \beta_i\\
\alpha{i+1}& \beta_{i+1}\\
\end{vmatrix}
=\pm 1.
\]
(b) If $[b';(\alpha_1,\beta_1),\ldots,(\alpha_n, \beta_n);b'']$ is a weighted arc then $b'\alpha_1+\beta+1=\pm 1$ and $b''\alpha_n+\beta_n=\pm 1$. (So for $i=1$ or $n$, $\beta_i=1$ or $\alpha_i-1$, and $b'$ and $b''$ can only take on the values $0 $ or $-1$.)
\end{lem}

The oriented orbit space $M^*$ together with the above collection of weights is called a \emph{weighted orbit space}. More generally, recall that a \emph{legally weighted simply-connected $3$-manifold} is an oriented simply-connected compact $3$-manifold $X^*$ along with the following data:
\begin{itemize}
	\item[(A)] an integer $a_i$ assigned to each boundary component of $X^*$,\\
	\item[(B)] a finite collection of points in $\mathrm{int}\, X^*$ with each assigned an integer $b_i=\pm 1$, and\\
	\item[(C)] a collection of weighted arcs and circles in $\mathrm{int}\, X^*$ as above and satisfying the criteria of Lemma~\ref{L:Fintushel_lemma}. To each weighted arc $A_i^*=[b';(\alpha_1,\beta_1),\ldots,(\alpha_n,\beta_n);b'']$ the integer $c_i=b''-b'$ is assigned.
\end{itemize}
At least one of the above collections must be nonempty and we require $\Sigma a_i+\Sigma b_i +\Sigma c_i=0$. It is shown in \cite{F1} that the weighted orbit space of an $\Ss^1$-action on a simply-connected $4$-manifold is legally weighted. 

It follows from Theorem \cite{F1} (7.1) and the validity of the Poincar\'e conjecture that, if $M^*$ contains no weighted circles, then any $\Ss^1$-action on  a simply-connected $4$-manifold $M$ extends to an action of $\mathsf{T}^2=\Ss^1\times\Ss^1$. As part of the proof of Theorem~B (cf. Section~\ref{section:dimension_4}), we determined all the possible orbit spaces of an isometric fixed-point homogeneous $\Ss^1$-action on a nonnegatively curved Riemanian $4$-manifold $M$. When $M$ is simply-connected, the orbit space contains no weighted circles and hence the $\Ss^1$-action must extend to a $\mathsf{T}^2$-action, answering affirmatively Question~\ref{question} in the case of a fixed-point homogeneous $\Ss^1$-action. We summarize this in the following corollary.

\begin{cor}
A fixed-point homogeneous isometric $\Ss^1$-action on a simply-connected nonnegatively curved $4$-manifold must be equivariantly diffeomorphic to a linear action on $\sphere^4$, $\CP^2$ or to an extendable action on $\sphere^2\times\sphere^2$ or $\CP^2\#\pm\CP^2$.
\end{cor}



\subsubsection{Equivariant plumbing} The equivariant plumbing of $2$-disk bundles over $2$-spheres is used in \cite{F1} to construct $4$-manifolds with $\Ss^1$-actions out of orbit space data. We will review this construction in this subsection. The basic building blocks will be $2$-disk bundles over $\sphere^2$ equipped with a given $\Ss^1$-action. First we show how to construct a $2$-disk bundle over $\sphere^2$ with Euler number $\omega$ equipped with certain $\Ss^1$-action and then we see how these disk bundles can be equivariantly plumbed together to obtain a given orbit space configuration (cf. \cite{F1} 4., 5.). 

  Write $\sphere^2=B_1\cup B_2$ as the union of its upper and lower hemispheres and consider polar coordinates on $B_i\times D_i^2$, $i=1,2$. Given relatively prime integers $u_i$ and $v_i$, define an $\Ss^1$-action on $B_i\times D_i$ by $\phi(r,\gamma, s, \delta)\mapsto(r, \gamma+u_i\phi,s,\delta+v_i\phi)$. If $u_2=-u_1$ and $v_2=-\omega u_1+v_1$ we obtain $Y_\omega=B_1\times D_1\cup_{G} B_2\times D_2$ via the equivariant pasting $G:\partial B_1\times D_1\rightarrow \partial B_2\times D_2$ given by $(1,\gamma, s,\delta)\mapsto(1,-\gamma,s, -\omega\gamma +\delta)$. The  $4$-manifold with boundary $Y_\omega$ is the $D^2$-bundle over $\sphere^2$ with Euler number $\omega$, i.e., $\omega$ is the self-intersection number of the zero section of $Y_\omega$.

Given $Y_{\omega_1}$ and $Y_{\omega_2}$ with $u_{2,1}=v_{1,2}$ and $v_{2,1}=u_{1,2}$ (or $u_{2,1}=-v_{1,2}$ and $v_{2,1}=-u_{1,2}$) we may equivariantly plumb $Y_{\omega_1}$ and $Y_{\omega_2}$ with sign $+1$ (sign $-1$) by identifying $B_{2,1}\times D_{2,1}$ with $B_{1,2}\times D_{1,2}$ by means of the equivariant diffeomorphism $(r,\gamma,s,\delta)\mapsto(s,\delta, r,\gamma)$ $((r,\gamma,s, \delta)\mapsto(s,-\delta,r,-\gamma))$. The resulting manifold, which we denote by $Y_{\omega_1}\square Y_{\omega_2}$, has an induced $\Ss^1$-action.

We may carry out these constructions also with $\mathsf{T}^2$-actions on $Y_\omega$ using integers $u_i$, $v_i$, $w_i$ and $t_i$ with
\[
\begin{vmatrix}
u_i	&	w_i\\
v_i	&	t_i\\
\end{vmatrix}
=\pm 1.
\]
The $\mathsf{T}^2$-action on $B_i\times D_i$ is given by $(\phi,\theta)(r,\gamma,s,\delta)\mapsto (r,\gamma+u_i\phi+w_i\theta,s,\delta+v_i\phi+t_i\theta)$. The glueing map $G$ defined in the preceding paragraph will be equivariant provided $w_2=-w_1$ and $t_2=-\omega w_1+t_1$. We may construct $Y_{\omega_1}\square Y_{\omega_2}$ with sign $+1$ and $\mathsf{T}^2$-equivariantly if $w_{2,1}=t_{1,2}$.

\subsubsection{Some examples.} We will now describe some of the disk bundles catalogued in \cite{F1} that we will use in our constructions. As described above, actions of $\Ss^1$ and $\mathsf{T}^2$ on $Y_\omega$ are determined by a matrix 
\[
\begin{pmatrix}
u_1	& u_2	& w_1	& w_2\\
v_1	&	v_2	& t_1	& t_2\\
\end{pmatrix}
\]
whose entries satisfy certain conditions. We will use the following disk bundles and actions (cf. \cite{F1}). We will assume that $\varepsilon=\pm 1$, $n$ is an arbitrary integer, and pairs $(\alpha, \beta)$ consist of relatively prime integers $0<\beta<\alpha$. \\

(c) If $b'\alpha+\beta=\pm 1$, $b''\alpha+\beta =\pm 1$, $\varepsilon'=\begin{vmatrix}1 & |b'|\\ \alpha & \beta\end{vmatrix}$, $\varepsilon'' = \begin{vmatrix}\alpha & \beta \\ 1 & |b''|\end{vmatrix}$
and $\omega=\varepsilon'\varepsilon''\begin{vmatrix}1 & |b'|\\ 1 & |b''|\end{vmatrix}$, then
\[
\begin{pmatrix}
\varepsilon\alpha	& -\varepsilon\alpha	& \varepsilon(\beta+n\alpha)	& -\varepsilon(\beta+n\alpha)\\
\varepsilon\varepsilon'	&	-\varepsilon\varepsilon''	&-\varepsilon\varepsilon'(|b'|+n)	& -\varepsilon\varepsilon''(|b''|+n)\\
\end{pmatrix}
\]
defines actions on $Y_\omega$ with $Y^*_\omega\cong D^3$ and a weighted arc $\bullet \longrightarrow \bullet $ with weights $[b';(\alpha,\beta);b'']$.
\\

(d)Let $\varepsilon',\varepsilon''=\pm 1$ and $\omega=-\varepsilon'-\varepsilon''$. Then
\[
\begin{pmatrix}
\varepsilon	& -\varepsilon	& \varepsilon n	& -\varepsilon n\\
-\varepsilon\varepsilon'	&	\varepsilon\varepsilon''	&-\varepsilon\varepsilon'(n+\varepsilon')	& \varepsilon\varepsilon''(n-\varepsilon'')\\
\end{pmatrix}
\]
describes actions on $Y_\omega$ with $Y_\omega^*\cong D^3$ with two isolated fixed-points with weights $\varepsilon'$ and $\varepsilon''$.
\\

(g) Suppose $b'\alpha' + \beta'=\pm 1$, $\varepsilon'=\begin{vmatrix} \alpha' & \beta' \\ 1 & |b'|\end{vmatrix}$ and $\omega=\varepsilon'\alpha'$. Then
\[
\begin{pmatrix}
\varepsilon	& -\varepsilon	& \varepsilon(|b'|+n)	& -\varepsilon(|b'|+n)\\
-\varepsilon\varepsilon'\alpha'	&	0	&\varepsilon\varepsilon'(\beta'+n\alpha')	& -\varepsilon\\
\end{pmatrix}
\]
defines actions on $Y_\omega$ and $Y^*_{\omega}$ with a fixed $D^2$ and  half a weighted arc $\longrightarrow \bullet$ with weights $(\alpha',\beta')$ and $b'$.
\\

(h) Let $\varepsilon'=\pm 1$ and $\omega=-\varepsilon'$. Then
\[
\begin{pmatrix}
\varepsilon	& -\varepsilon	& \varepsilon	& \varepsilon n\\
-\varepsilon\varepsilon'	&	0	&-\varepsilon\varepsilon'(n+\varepsilon')	& -\varepsilon\\
\end{pmatrix}
\]
describes actions on $Y_\omega$ with $Y_\omega^*\cong D^3$ with an isolated fixed point with weight $\varepsilon'$ and a fixed $D^2$.
\\

(i) Let $\delta=\pm 1$. Then
\[
\begin{pmatrix}
\varepsilon	& -\varepsilon	& n	& - n\\
0	&	0	&\delta	& \delta \\
\end{pmatrix}
\]
describes actions on $Y_0$ with $Y_0^*\cong D^3$ with two fixed $2$-disks.
\\

(j) For $\omega$ arbitrary and $\delta=\pm 1$ actions on $Y_\omega$ are defined by 
\[
\begin{pmatrix}
0	& 0	& \delta	& - \delta \\
\varepsilon	&	\varepsilon	& n	& -\omega\delta +n \\
\end{pmatrix}
\]
and $Y^*_{\omega}\cong\sphere^2\times I$ with $E^*\cup F^* = F^*=\sphere^2\times 0$ with weight $\omega$.



\subsubsection{Computation of the intersection form} In \cite{F1} there is a catalog of different disk-bundles with $\Ss^1$- and $\mathsf{T}^2$-actions realizing different basic orbit space configurations. If $M^*$ contains no weighted circles, these disk bundles may be plumbed together to construct a $4$-manifold $R$  whose orbit space $R^*$ is a particular subset of $M^*$. We will outline the construction of $R$ and then recall the recipe given in \cite{F1} for computing the intersection form of $M$ out of the intersection form of $R$ (cf. \cite{F1}, 5.,8.).

Let $S^{*}_{1},\ldots,S^{*}_{t}$ be the collection of weighted sets in $M^*$ other than the weighted circles, with the weighted boundary components of $M^*$, if any, listed at the end. For each $i=1,\ldots,t-1$ let $\gamma^*_i$ be an arc in $M^*$ joining $S^*_{i}$ to $S^*_{i+1}$ such that the interior of the arc lies in the regular orbit stratum $P^*$ and such that if $S_{i}^*$ is a weighted arc, $\gamma_i^*$ begins at the endpoint of $S^*_{i}$, and if $S^*_{i+1}$ is a weighted arc, then $\gamma^*_i$ ends at the initial point of $S^*_{i+1}$. Let $R^*$ be a regular neighborhood of $\bigcup S^*_{i}\cup\bigcup\gamma^*_i$. By equivariantly plumbing  disk bundles $Y_{\omega_i}$ listed in \cite{F1} (with each plumbing of sign $+1$) one can construct a $4$-manifold $R$ with $\Ss^1$-action and weighted orbit space isomorphic to $R^*$. Moreover, this action extends to a $\mathsf{T}^2$-action (cf. Lemma 4.7 in \cite{F1}).

Let $M$ be a simply-connected $4$-manifold with a smooth $\Ss^1$-action such that $M^*$ contains no weighted circles. We now recall how to recover the intersection form $Q_M$ of $M$ out of the set  $R^*$. Let $R$ be the $4$-manifold  with $\Ss^1$-action and weighted orbit space isomorphic to $R^*$. Then $R$ is the result of an equivariant linear plumbing 
 $$\xygraph{
!{<0cm,0cm>;<1cm,0cm>:<0cm,1cm>::}
!{(0,0) }*+{\bullet}="a"
!{(0,-.3) }*+{\omega_1}
!{(1,0) }*+{\bullet}="b"
!{(1,-.3) }*+{\omega_2}
!{(2,0) }*+{\ldots}="c"
!{(3,0) }*+{\bullet}="d"
!{(3,-.3) }*+{\omega_t}
"a"-"b"^{} "b"-"c"^{} "c"-"d"^{}
}$$

If $\partial M^*$ has $m$ components and $(F^*-\partial M^*)\cap R^*$ contains $l$ points then $t=2m+l-1$. The intersection matrix $B_0$ of the plumbing $R$ is the $t\times t$ matrix given by
\[
	\left[ B_0\right]_{ij}=
	\begin{cases}
		\omega_i,	&	i=j,\\
		1,		&	i=j\pm1,\\
		0,		&	\text{otherwise},
	\end{cases}
\]
since each plumbing has sign $+1$. 

Given a square matrix $B$,  we will denote by $B^-$ the matrix obtained after removing the last row and column from $B$. It is shown in  \cite{F1} that the intersection form $Q_M$ of $M$ is $B_0^-$.


\subsection{Proof of Theorem~C}
\label{Section:computations}
To prove Theorem~C, we will determine the possible legally weighted orbit spaces of a simply-connected nonnegatively curved Riemannian $4$-manifold $M$ with an isometric fixed-point homogeneous $\Ss^1$-action. We will also identify $M$ out of the orbit space data following the constructions described in Section~\ref{Section:orbit_space}. By Theorem~\ref{T:D4:S1:KSY-classification} (2),  $M$ is diffeomorphic to $\sphere^4$, $\CP^2$, $\sphere^2\times\sphere^2$ or $\CP^2\#\pm\CP^2$. It is well known that $\chi(M)=\chi(\Fix(M,\Ss^1))$ (cf. \cite{Ko}) and, since the action is fixed-point homogeneous, $\Fix(M,\Ss^1)$ must contain a $2$-sphere. Hence we have the following possible fixed-point sets:
\begin{equation}
\label{L:fixed-point_sets}
\Fix(M,\Ss^1)=
	\begin{cases}
		 \sphere^2 				&\text{if $M$ is $\sphere^4$.}\\
		 \sphere^2\cup\{p\} 		&\text{if $M$ is $\CP^2$.}\\
		 \sphere^2\cup\sphere^2 	&\text{if $M$ is $\sphere^2\times\sphere^2$ or $\CP^2\pm\CP^2$.}\\
		 \sphere^2\cup\{p',p''\}	&\text{if $M$ is $\sphere^2\times\sphere^2$ or $\CP^2\pm\CP^2$.}
	\end{cases} 
\end{equation}

By our analysis in Section 4, the orbit space of an isometric fixed-point homogeneous circle action on a simply-connected nonnegatively curved manifold $M$ does not contain any weighted circles. Hence we restrict our analysis to these orbit spaces. Observe that there cannot be any exceptional orbits unless $\Fix(M,\Ss^1)$ contains two isolated fixed points. Hence, when $\Fix(M,\Ss^1)$ contains at most one isolated fixed point, corresponding to $\Fix(M,\Ss^1)=\sphere^2$ or $\sphere^2\cup\{p\}$, we may dispense with the geometric assumptions, since the orbit space structure itself prevents the existence of any weighted circles. It follows then that any fixed-point homogeneous circle action on $\sphere^4$ or $\CP^4$ is equivariantly diffeomorphic to a linear action. However, when $F$ contains two isolated fixed points we will explicitly assume that the orbit space contains no weighted circles.

We will prove three propositions, corresponding to (3)--(5) in Theorem~C. 
Parts (1) and (2) follow from the comments at the beginning of this subsection.
We will proceed as follows. Given a fixed-point set $F$  we will construct $R$ as in Section~\ref{Section:orbit_space} using the pieces we have described therein. We will then identify $M$ by computing its intersection form $Q_M$ following the recipe in Section~\ref{Section:orbit_space}.


\begin{case}$\Fix(M,\Ss^1)=\sphere^2\cup \sphere^2$.
\begin{prop}
\label{P:S2_S2}
Let $M$ be a simply-connected smooth $4$-manifold with a smooth $\Ss^1$-action. If $\mathrm{Fix}(M,\Ss^1)=\sphere^2\cup\sphere^2$, then $M$ is equivariantly diffeomorphic to $\CP^2\#-\CP^2$ or $\sphere^2\times\sphere^2$ with an extendable action.
\end{prop}

\begin{proof}
We construct $R$ using bundles $Y_{\omega_1}$, $Y_{\omega_2}$ and $Y_{\omega_3}$ with actions (j), (i) and (j), respectively. Observe that $\omega_2=0$, so the plumbing $Y_{\omega_1}\Box Y_{\omega_2}\Box Y_{\omega_3}$ has intersection form
\begin{equation*}
B_0 =
	\begin{bmatrix}
	 \omega_1	&	1 	&	0\\
	1			&	\omega_2	&	1\\
	0			&	1	&	\omega_3
	\end{bmatrix}
=
\left[
	\begin{array}{ccc}
	 \omega_1	&	1 	&	0\\
	1			&	0	&	1\\
	0			&	1	&	\omega_3
	\end{array}
\right] .
\end{equation*}
 
The intersection form of $M$ is then $B_0^-$, i.e.,
\[
Q_M =\left[
	\begin{array}{cc}
	\omega_1	& 	1\\
	1		&	0
	\end{array}
\right] .
\]
Now we show that $Q_M$ is equivalent to the intersection form of $\CP^2\#-\CP^1$, if $\omega_1$ is odd, and to the intersection form of $\sphere^2\times\sphere^2$, if $\omega_1$ is even.

Recall that the operation of adding an integral constant $k$ times row $i$ to row $j$ and then that constant times column $i$ to column $j$ preserves the congruence class over $\Int$ of an integral matrix. We call this an \emph{elementary operation} and will keep track of it by denoting it by $(i,j; k)$. 
We have
\[
	\begin{bmatrix}
	\omega_1	& 	1\\
	1		&	0
	\end{bmatrix}
\xrightarrow{(2,1;\pm1)}
	\left[
	\begin{array}{cc}
	\omega_1\pm2	& 	1\\
	1		&	0
	\end{array}
	\right].
\]
Thus, after repeated application of the elementary operation $(2,1;\pm 1)$ to 
$\begin{bmatrix} \omega_1 & 1\\ 1 & 0\end{bmatrix}$
we have 
\[
Q_M\cong	\begin{bmatrix}
	\omega_1\ (\text{mod $2$})	& 1\\
	1	& 0
	\end{bmatrix} .
\]

When $\omega_1$ is even, we have
\[
Q_M\cong	\begin{bmatrix}
	0	& 1\\
	1	& 0
	\end{bmatrix}, 
\]
which is the intersection form of $\sphere^2\times\sphere^2$.

When $\omega_1$ is odd, we have
\[
	\begin{bmatrix}
	1	& 	1\\
	1		&	0
	\end{bmatrix}
\xrightarrow{(1,2;-1)}
	\begin{bmatrix}
	1	& 	0\\
	0	&	-1
	\end{bmatrix},
\]
which is the intersection form of $\CP^2\#-\CP^2$.

\end{proof}

\begin{rem} Proposition~\ref{P:S2_S2} and its proof show that the fact that $\CP^2\#\CP^2$ does not admit any smooth circle action with fixed-point set the union of two $2$-spheres is a purely topological phenomenon. Under the additional condition of nonnegative curvature, this follows from the Double Soul Theorem, which implies that $M^4$ is an $\sphere^2$-bundle over $\sphere^2$ and hence $M^4$ must be $\sphere^2\times\sphere^2$ or $\CP^2\#-\CP^2\cong\sphere^2\tilde{\times}\sphere^2$. 
\end{rem}

\end{case}


\begin{case}$\Fix(M,\Ss^1)=\sphere^2\cup\{p',p''\}$. We split this case into two subcases, depending on whether or not there are any orbits with finite isotropy.
\\

\noindent \textbf{No finite isotropy.} Suppose first there are no orbits with finite isotropy. 

\begin{prop} Let $M^4$ be a simply-connected smooth $4$-manifold with a smooth $\Ss^1$-action without finite isotropy. If $\Fix(M,\Ss^1)=\sphere^2\cup\{p',p''\}$, then $M$ is equivariantly diffeomorphic to $\CP^2\#\pm\CP^2$ with an extendable action.
\end{prop}

\begin{proof}

To compute the intersection form $Q_M$ of $M$ we first construct $R$ using the bundles $Y_{\omega_1}$ with action (d), $Y_{\omega_2}$ with action (h) and $Y_{\omega_3}$ with action (j). The intersection form of the plumbing $Y_{\omega_1}\Box Y_{\omega_2}\Box Y_{\omega_3}$ is
\[
	B_0=
	\begin{bmatrix}
	\omega_1 & 	1 			& 0\\
	1		&	\omega_2		& 1\\
	0		&	1			& \omega_3
	\end{bmatrix}.
\] 
 Then the intersection form of $M$ is given by $B_0^-$, i.e.
 \[
 	Q_M=
	\begin{bmatrix}
	\omega_1		& 	1\\
	1			& \omega_2
	\end{bmatrix}.
 \]
 We now determine $\omega_1$ and $\omega_2$.
Let $\varepsilon_1',\ \varepsilon_1'' = \pm 1$. Then $\omega_1=-\varepsilon_{1}'-\varepsilon_{1}''$, coming from action $(d)$. On the other hand, for $Y_{\omega_2}$ we have $\omega_2=-\varepsilon_{2}'$, where $\varepsilon_2'=\pm 1$. In order to plumb these two bundles together, we need $\varepsilon_{1}''=\varepsilon_{2}'$. Hence $\omega_2=-\varepsilon_{2}'=-\varepsilon_{1}''$. 

To obtain the conclusion of the Proposition, compute the possible intersection forms  $Q_M$ in terms of  $\omega_1=-\varepsilon_{1}'-\varepsilon_{1}''$ and $\omega_2=-\varepsilon_{1}''$ and apply the elementary operations $(2,1;1)$, when $\varepsilon_{1}'=\varepsilon_{1}''=1$, and $(2,1;-1)$ in the remaining cases.

 \end{proof}

\end{case}


\noindent \textbf{Finite isotropy.} Suppose there are points with finite isotropy. 

\begin{prop} Let $M^4$ be a simply-connected smooth $4$-manifold with a smooth $\Ss^1$-action with $\Fix(M,\Ss^1)=\sphere^2\cup\{p',p''\}$ and a weighted arc with finite isotropy $\Int_k$. Then $M$ is equivariantly diffeomorphic to one of the following:
\begin{itemize}
			\item[(1)]  $\CP^2\#\CP^2$ with an extendable action with finite isotropy $\Int_2$.\\
			\item[(2)]  $\CP^2\#-\CP^2$ with an extendable action with finite isotropy $\Int_k$, $k$ odd.\\
			\item[(3)]  $\sphere^2\times\sphere^2$ with an extendable action with finite isotropy $\Int_k$, $k$ even.\\
		\end{itemize}

\end{prop}

\begin{proof}
Let  $[b'; (\alpha_1, \beta_1);b'']$ be the weighted arc. In this case $\beta_1=1$ or $\alpha_1-1$ and $b' $ and $b''$ can only take on the values $0 $ or $-1$ (cf. Lemma~3.5 in \cite{F1}). We will use actions (c), (g) and (j). Recall that, to each weighted arc $[b';(\alpha_1,\beta_1),\ldots,(\alpha_n,\beta_n);b'']$, the integer $c=b''-b'$ is assigned (cf. \cite{F1}(5.2)(c)). 
For the orbit space to be legally weighted, we must have $a+c=0$, where $a$ is the weight of the boundary $2$-sphere, so $a=-c$. The following table lists the possible combinations of weights.
\begin{center}
\begin{tabular}{| r | r | r | r | }
\hline
	$b'$		& $b''$ 	& $c=b''-b'$	& $a$ 	\\ \hline\hline
	$0$		& $0$	& $0$		& $0$	\\ \hline	
	$0$		& $-1$	& $-1$		& $1$	\\ \hline
	$-1$		& $0$	& $1$		& $-1$	\\ \hline 
	$-1$		& $-1$	& $0$		& $0$	\\ \hline			
\end{tabular}
\end{center}

\noindent\textbf{1.} The first piece we need is a bundle $Y_{\omega_1}$ with action (c) as in Section~\ref{Section:orbit_space}. We have

\[
\pm 1 = \varepsilon'_1 = \left| \begin{array}{cc}
	1 	& 	|b'| \\
	\alpha 	&	\beta
	\end{array} \right|
	= \beta -\alpha|b'_{1}| 
	= 
	\begin{cases}
	\beta, 	& 	\mathrm{if}\ b_{1}'=0;\\
	\beta-\alpha, \ 	&	\mathrm{if}\ b'_{1}=-1.
	\end{cases}
\]
\bigskip

\[
\pm 1 = \varepsilon''_1 
	 =  \left| \begin{array}{cc}
	\alpha 	& 	\beta \\
	1 	&	|\beta''|
	\end{array} \right|
 	= \alpha|b''_{1}| - \beta  
	= 
	\begin{cases}
	-\beta, 	     & 	\mathrm{if}\ |b''_{1}|=0;\\
	\alpha-\beta, &	\mathrm{if}\ b''_{1}=-1.
	\end{cases}
\]
We also have 
\[
	\omega_1 = \varepsilon'_1\varepsilon''_1
	\begin{vmatrix}
	1	& 	|b'_{1}|\\
	1	&	|b''_{1}|
	\end{vmatrix}\\
	 =
	  \varepsilon'_1\varepsilon''_1(|b''_{1}|=|b'_{1}|).
\]

We have the following possible combinations:
\\

 \begin{center}
\begin{tabular}{| r | r | r | r | r | }
\hline
	$b'_{1}$	& $b''_{1}$ & $\varepsilon'_1$	& $\varepsilon''_1$	& $\omega_1$ 	\\ \hline \hline
	$0$		& $0$	& $\beta$			& $-\beta$			& $0$	\\ \hline	
	$0$		& $-1$	& $\beta$			& $\alpha-\beta$	& $\beta(\alpha-\beta)$\\ \hline
	$-1$		& $0$	& $-(\alpha-\beta)$	& $-\beta$			& $-\beta(\alpha-\beta)$ \\ \hline 
	$-1$		& $-1$	& $-(\alpha-\beta)$	& $\alpha-\beta$	& $0$\\ \hline			
\end{tabular}
\end{center}
\bigskip

\textbf{Case:} $(b'_1,b''_1)=(0,0)$. We have $\beta=\varepsilon'_{1}=\pm1$. Recall that $\beta=1$ or $\alpha - 1$. Hence $1=\beta=\varepsilon'_1$ and $\varepsilon''_1=-1$.
\\

\textbf{Case:} $(b'_1,b''_1)=(0,-1)$. We have $\varepsilon'_1=\pm 1=\beta>0$ so $\varepsilon'_1=\beta =1$. Hence 
\[
	\pm1=\varepsilon''_1 =\alpha-\beta = \alpha-1.
\]

We have $\alpha\geq 2$ so $\alpha-1\geq 1>0$. Hence $\varepsilon''_1=+1$. Hence $\alpha-1=1$ so $\alpha=2$.
\\

\textbf{Case:} $(b'_1,b''_1)=(-1,0)$. Recall that $\beta$ takes on the values $1$ or $\alpha-1$. We have

\[
\pm 1  =  \varepsilon'_1 = -(\alpha-\beta)=
\begin{cases}
	 -(\alpha-1), 	& \text{if $\beta=1$};\\
	 -1,			& \text{if $\beta=\alpha-1$}.
\end{cases}
\]

\[
\pm 1  =  \varepsilon''_1 = -\beta=
\begin{cases}
	 -1, 			& \text{if $\beta=1$};\\
	 -(\alpha-1),	& \text{if $\beta=\alpha-1$}.
\end{cases}
\]
It follows from these equations that $\varepsilon'_1=\varepsilon''_1=-1$ and $\alpha=2$, $\beta=1$.
\\

\textbf{Case:} $(b',b'')=(-1,-1)$. We have

\begin{eqnarray*}
\pm 1 & = & \varepsilon'_1 = -(\alpha-\beta)=-\varepsilon''_1.
\end{eqnarray*}

Recall that $\beta=1$ or $\alpha-1$. In both cases the equation above implies that $\varepsilon'_1=-1$ and $\varepsilon''_1=+1$. Observe that any $\alpha\geq 2$ is possible. 
\\

We update the table of weights in the previous page and obtain the following list of weights.

 \begin{center}
\begin{tabular}{| r | r | r | r | r | r | r |}
\hline
	$b'_{1}$	& $b''_{1}$ & $\varepsilon'_1$	& $\varepsilon''_1$	& $\omega_1$ &$\alpha$&$\beta$\\ \hline \hline
	$0$		& $0$	& $1$			& $-1$			& $0$		& $k\geq 2$ & $k-1$\\ \hline	
	$0$		& $-1$	& $1$			& $1$			& $1$		& $2$	& $1$\\ \hline
	$-1$		& $0$	& $-1$			& $-1$			& $-1$		& $2$	& $1$ \\ \hline 
	$-1$		& $-1$	& $-1$			& $1$			& $0$		& $k\geq 2$& $k-1$\\ \hline
\end{tabular}
\end{center}
\bigskip

\noindent \textbf{2.} Now we deal with piece $2$, coming from bundle $Y_{\omega_2}$ with action (g). We have weights $b'_2$, $\alpha'_2$ and $\beta'_2$. In order to plumb $Y_{\omega_1}$ and $Y_{\omega_2}$ we need $\alpha_1=\alpha'_2$, $\beta_1=\beta'_2$ and $b'_2=b''_1$. The subscript $i$ denotes the bundle $Y_{\omega_i}$ to which each weight belongs. We also have
\[
\varepsilon'_2 	 = \begin{vmatrix}\alpha'_2 & \beta'_2\\ 1 & |b'_2|\end{vmatrix}
			 = \begin{vmatrix}\alpha_1 & \beta_1\\ 1 & |b''_1|\end{vmatrix}
			 = \varepsilon_1''.
\]
Since $\omega_2=\varepsilon'_2\alpha'_2$, we have
\[
\omega_2=\varepsilon_1''\alpha_1.
\]

Hence we have the following combinations:

 \begin{center}
\begin{tabular}{| r | r | r | r | r | r | r | r | }
\hline
	$b'_{1}$	& $b''_{1}$ & $\varepsilon'_1$	& $\varepsilon''_1$	& $\omega_1$ &$\alpha$		&$\beta$	& $\omega_2=\varepsilon''_1\alpha$ \\ \hline \hline
	$0$		& $0$	& $1$			& $-1$			& $0$		& $k\geq 2$ 	& $k-1$	&	$-k$	\\ \hline	
	$0$		& $-1$	& $1$			& $1$			& $1$		& $2$		& $1$	&	$2$	\\ \hline
	$-1$		& $0$	& $-1$			& $-1$			& $-1$		& $2$		& $1$ 	&	$-2$	\\ \hline 
	$-1$		& $-1$	& $-1$			& $1$			& $0$		& $k\geq 2$	& $k-1$	& 	$k$	\\ \hline
\end{tabular}
\end{center}
\bigskip

\noindent\textbf{3.} The last piece we need is a bundle $Y_{\omega_3}$ with action (j). The intersection form of the plumbing $Y_{\omega_1}\Box Y_{\omega_2}\Box Y_{\omega_3}$ is
\[
	B_0=
	\begin{bmatrix}
	\omega_1 & 	1			& 	0\\
	1		&	\omega_2		&	1\\
	0		&	1			& \omega_3\\
	\end{bmatrix}.
\]
Hence the intersection form $Q_M$ of $M$ is $B_0^-$, i.e., 
\[
	Q_M=
	\begin{bmatrix}
	\omega_1 & 1\\
	1		& \omega_2
	\end{bmatrix}.
\]

When $b_1'=0$ and $b''_1=-1$, we have
\[
	Q_M=
	\begin{bmatrix}
	1 & 1\\
	1 & 2
	\end{bmatrix}
	\xrightarrow{(1,2;-1)}
	\begin{bmatrix}
	1	& 	0\\
	0	&	1
	\end{bmatrix},
\]
which is the intersection form of $\CP^2\#\CP^2$.

When $b_1'=-1$ and $b''_1=0$, we have
\[
	Q_M=
	\begin{bmatrix}
	-1 & 1\\
	1 & -2
	\end{bmatrix}
	\xrightarrow{(1,2;1)}
	\begin{bmatrix}
	-1	& 	0\\
	0	&	-1
	\end{bmatrix},
\]
which is the intersection form of $-\CP^2\#-\CP^2$.
Observe that in these two cases (which are the same up to orientation) we can only have isotropy $\Int_2$.

When $b'_1=b''_1=0$, we have 
\[
	Q_M=
	\begin{bmatrix}
	0 & 1\\
	1& -k
	\end{bmatrix}
\]
for $k\geq 2$. 

After repeated applications of the elementary operation $(1,2;1)$ we have
\[	
	Q_M\cong
	\begin{bmatrix}
	0 & 1\\
	1 & -k\ \mathrm{mod}\ 2
	\end{bmatrix}.
\]
When $k$ is even, we have
\[
Q_M\cong
	\begin{bmatrix}
	0 & 1\\
	1 & 0
	\end{bmatrix}
\]
which is the intersection form of $\sphere^2\times\sphere^2$. When $k$ is odd, we have
\[
Q_M\cong
	\begin{bmatrix}
	0 & 1\\
	1 & 1
	\end{bmatrix}
	\xrightarrow{(2,1;-1)}
	\begin{bmatrix}
	-1	& 	0\\
	0	&	1
	\end{bmatrix}.
\]
which is the intersection form of $-\CP^2\#\CP^2$.


When $b'_1=b''_1=-1$, we have 
\[
	Q_M=
	\begin{bmatrix}
	0 & 1\\
	1& k
	\end{bmatrix}
\]
for $k\geq 2$. An analogous argument to the one we used when $b'_1=b''_1=0$, now using the elementary operation  $(1,2;1)$, yields the intersection form of $\sphere^2\times\sphere^2$, when $k$ is even, and of $-\CP^2\#\CP^2$, when $k$ is odd.

\end{proof}



\bibliographystyle{amsplain}

\providecommand{\bysame}{\leavevmode\hbox to3em{\hrulefill}\thinspace}
\providecommand{\MR}{\relax\ifhmode\unskip\space\fi MR }
\providecommand{\MRhref}[2]{%
  \href{http://www.ams.org/mathscinet-getitem?mr=#1}{#2}
}
\providecommand{\href}[2]{#2}

\end{document}